\RequirePackage{fix-cm}
\documentclass[smallextended]{svjour3pre}%-without-header}                   % onecolumn (standard format)
\smartqed  % flush right qed marks, e.g. at end of proof
\usepackage{bm}
\usepackage{amsmath,amssymb,amsbsy,latexsym}
\usepackage{graphicx}
\usepackage{mathptmx}      % use Times fonts if available on your TeX system
\usepackage{epstopdf}
\usepackage{subfigure}
\usepackage{xcolor}
\usepackage{verbatim}
\usepackage[percent]{overpic}

\journalname{BIT}
\date{\today}
\numberwithin{equation}{section}
\numberwithin{figure}{section}
\numberwithin{table}{section}
\newcommand{\C}{\mathbb C}% complex numbers
\newcommand{\R}{\mathbb R}

\newcommand{\divdif}[2]{#2[#1]}
\DeclareMathOperator{\vecspan}{span}

\newcommand{\rr}{{\rm r}}
\newcommand{\ttol}{{\rm tol}}
\newcommand\tjnew[1]{{#1}}
\newcommand\phif\varphi
\newcommand{\dd}{\mathrm{d}}
\newcommand{\ee}{{\rm e}}
\newcommand{\ii}{{\rm i}}
\newcommand{\Kry}{\mathcal{K}}

\newcommand\Era{\mathrm{Err}_a}
\newcommand\Erone{\mathrm{Err}_1}

\newcommand\Er{\mathrm{Err}}
\newcommand\Vmj{V_m^{[j]}}
\newcommand\Tmj{T_m^{[j]}}
\newcommand\Smj{S_m^{[j]}}
\newcommand\Lmj{L_m^{[j]}}
\newcommand\mw[1]{w^{[#1]}}
\newcommand\mv[1]{v^{[#1]}}
\newcommand\mvmone[1]{v^{[#1]}_{m+1}}

\newcommand\nsit{n_{\text{sites}}}

\newcommand\Tcor{{T}^+}
\newcommand\Vcor{{V}^+}

\newcommand\Scor{{S}^+}

\newcommand\sca{\sigma}
\newcommand\Opgen{A}

\begin{document}

\sloppy
\title{Computable upper error bounds
       for Krylov approximations to matrix exponentials
       and associated $ \phif $-functions}
% of skew-Hermitian matrices}
\titlerunning{Computable upper error bounds
       for Krylov approximations}
\author{Tobias Jawecki \and Winfried Auzinger \and Othmar Koch}
\authorrunning{T.\ Jawecki et\ al.}
\institute{
Tobias Jawecki \and Winfried~Auzinger \at
Institut f{\"u}r Analysis und Scientific Computing, Technische Universit{\"a}t Wien \\
Wiedner Hauptstrasse 8--10/E101, A-1040 Wien, Austria \\
\email{tobias.jawecki@tuwien.ac.at, w.auzinger@tuwien.ac.at}
\and
Othmar~Koch \at
Fakult{\"a}t f{\"u}r Mathematik, Universit{\"a}t Wien \\
Oskar-Morgenstern-Platz 1, A-1090 Wien, Austria \\
\email{othmar@othmar-koch.org}
}

\maketitle

\begin{abstract}
An a posteriori estimate for the error of a standard Krylov approximation
to the \tjnew{matrix exponential} is derived.
\tjnew{The estimate is based on the defect (residual) of the Krylov approximation and
is proven to constitute a rigorous upper bound on the error,
in contrast to existing asymptotical approximations.}
It can be computed economically in the underlying Krylov space.
In view of time-stepping applications, assuming that
the given matrix is scaled by a time step, it is shown that
the bound is asymptotically correct (with an order related to the
dimension of the Krylov space) for the time step tending to zero.
\tjnew{This means that the deviation of the error estimate from the true error
tends to zero faster than the error itself. }
Furthermore, this result is extended to
Krylov approximations of $ \phif $-functions
and to improved versions of such approximations.
%\tjnew{Most of our results apply to general matrices, but our main interest concerning practical applications is in skew-Hermitian matrices which appear in the numerical treatment
%of Schr{\"o}dinger equations.}
The accuracy of the derived bounds is demonstrated by examples
and compared with different variants known from the literature,
which are also investigated more closely.
Alternative error bounds are tested on examples, in particular
a version based on the concept of effective order.
For the case where the matrix exponential is used
in time integration algorithms, a step size selection
strategy is proposed and illustrated by experiments. 
\end{abstract}

%\x{Keine Kasterl am Ende von Beweisen?}

\newcommand\wa[1]{#1}%\textcolor{blue}{#1}}
\newcommand\tj[1]{\textcolor{red}{#1}}
\newcommand\ok[1]{\textcolor{magenta}{#1}}

\section{Introduction}

We consider Krylov approximations to the matrix exponential function
for the purpose of the solution of a linear, homogeneous system
of differential equations
\begin{equation*}
\psi'(t)=M\psi(t), ~~  \psi(0)=\psi_0,
\qquad \psi(t) = \ee^{tM} \psi_0.
\end{equation*}
\tjnew{The complex matrix $M$ commonly results from the discretization of a partial differential equation.}
%Our main focus is on equations of Schr{\"o}dinger type,
%where $M=-\ii\,H$ with \tjnew{a Hermitian} matrix $H$,
%which commonly arise after spatial discretization of linear Schr{\"o}dinger
%equations.
%\tjnew{Our analysis also covers the case of Hermitian matrices
%and the dissipative case is also investigated by numerical experiments.
%However we do not include non-normal matrices in our considerations.
%Still, results that do not require special assumptions on the system matrix
%are formulated in fully general terms.}
In this work we present new results for precise a~posteriori error estimation,
which also extend to the evaluation of so-called $ \phif $-functions.
The application of these estimates for the purpose of time propagation
is also discussed and illustrated.
\tjnew{Theoretical results are verified by numerical experiments, which are classified into
Hermitian (dissipative), skew-Hermitian (Schr{\"o}dinger-type) and general non-normal problems.}

\paragraph{Overview on existing approaches and results.}
The approximate evaluation of large matrix exponential functions
is a topic which has been extensively treated in the numerical
analysis literature, for basic reference see e.g.~\cite{golloa89,molloa03}.
A standard approach is to project the given
matrix~$ M $ to a low-dimensional Krylov space via \tjnew{Arnoldi or Lanczos iteration},
%constructed from
%increasing powers of~$M$ realized by a short recursion in the
%\text{(anti-)Hermitian} Lanczos method~\cite{saad03},
and to directly exponentiate the projected small matrix.
A first mention of the Lanczos approach
can be found in~\cite{parlig86}, where it is also
recognized that for the method
to perform satisfactorily, the time-steps have to be controlled.
However, the control mechanism from~\cite{parlig86} is not
very elaborate and is based on a series expansion of the error, which
is only valid in the asymptotic regime, see for instance~\cite{niewri12}.
For discretizations of parabolic problems, \cite{galsaa92} uses an error estimator to choose
the step-size, this approach is improved in~\cite{steley96} and has been
generalized in~\cite{mohcar10}. Notably, in the latter reference a strict error bound is
used to estimate the time-step instead of asymptotic techniques. 
%It is furthermore observed that
%in contrast, strategies based on Taylor expansion tend to overestimate the error.
\tjnew{It is argued in~\cite{mohcar10} that the
strategy from~\cite{mohcar10} performs better than~\cite{mohauer06} and
better in turn than~\cite{parlig86}.}

A first systematic study of Krylov-based methods
for the matrix exponential function was
given in~\cite{saad92s}. The error
is analyzed theoretically, yielding both a priori and
computable a posteriori estimates. The analysis there relies on
approximation theory and yields a priori error bounds which are
asymptotically optimal in the dimension of the Krylov subspace in important situations. The analysis moreover
implies correction schemes to lift the convergence order which are cheap to compute based on
the already available information. The error expansion also
suggests a posteriori error estimators resorting to the
leading error term.
%, which again involve the exponential of a small matrix.
This approach relies on the assumption of the sufficiently rapid decay of the
\tjnew{series representation of the error}. A recent generalization of this work together
with a more rigorous justification is given in~\cite{chinesen17}.
\tjnew{For early studies of a~priori error estimates see also \cite{DK89,DK95}.}

A thorough theoretical analysis of the error of Krylov methods
for the exponential of \tjnew{a Hermitian} or skew- (anti-) Hermitian matrix
was given in~\cite{hoclub97}.
The analysis derives an asymptotic error expansion and
shows superlinear error decay in the dimension~$ m $ of
the approximation subspace for sufficiently large~$ m$.
\tjnew{These results are further improved in~\cite{becrei09}.}
%which in the skew-Hermitian case that we
%are confronted with holds only for sufficiently large subspaces,
In~\cite{hoclub97}, a~posteriori error estimation is also discussed.
This topic is furthermore addressed in~\cite{lubich08}.
There, the Krylov approximation method is interpreted as a Galerkin method, whence an
error bound can be obtained from an error representation for this variational approximation.
This yields a computable estimate
via a quadrature approximation
of the error integral involving the defect of the numerical
approximation. The a~priori error analysis reveals
a step-size restriction for the convergence of the method,
which is less stringent when the subspace dimension is larger.

Further work in the direction of controlling the Lanczos process through
information gained from the defect is given in~\cite{botchevetal13}.
The defect is a scalar %(time-dependent)
multiple of the successive Krylov vector
arising in the iteration and can be evaluated efficiently.
If the error is approximated by a Galerkin approach, the resulting estimator
corresponds to the difference of two Lanczos iterates. For the
purpose of practical error estimation, in~\cite{botchevetal13} it is seen as preferable to
continue the original Krylov process.
\tjnew{Some other defect-based upper bounds for the error
of the matrix exponential are given in~\cite{chinesen17}, including a closer analysis of the error estimate of~\cite{saad92s}.
These results still require some a~priori information on the matrix spectrum.
}

Various improved methods for computing the matrix
exponential function are given in the literature, for example restarted methods, deflated restarting methods
or quadrature based restarting methods, see~\cite{AEEG08a},~\cite{EEG11}, and~\cite{FGS14a}.

%\tjnew{Rational Krylov subspaces are widely used for eigenvalue problems, see also~\cite{Ru84}.}
It has also been advocated in~\cite{eshhoc06} to use preconditioning in the
Lanczos method by a shifted inverse in order to get a good approximation
of the leading invariant subspaces.
\tjnew{The shift-and-invert approach 
(a specific choice to construct a rational Krylov subspace)
for the matrix exponential function was
introduced earlier in~\cite{mornov04}.}
However, the choice of the shift is
critical for the success of this procedure.
 This strategy amounts to a
transformation of the spectrum which grants a convergence speed which
is independent of the norm of the given matrix.
In~\cite{eshhoc06}, a~posteriori
error estimation based on the asymptotical expansion of the error
is advocated as well.
We note that our results do not immediately carry over to the shift-and-invert approach, see Remark~\ref{SIkrylov}.

\paragraph{Overview on present work.}

In Section~\ref{sec:problem} we introduce the Krylov approximation
and the integral representation of the approximation error in terms of its defect.
\tjnew{In Section~\ref{sec:theorypart}
we derive a new computable upper bound for the error
by using data available from the Krylov process
with negligible additional computational effort (Theorem~\ref{theorem:upperboundexp}).}
This upper bound is cheap to evaluate and update on the fly
during the Lanczos iteration.
It is also asymptotically correct,
i.e., for $ t \to 0 $ the error of the error estimator tends
to zero faster asymptotically than the error itself. % (Proposition~\ref{prop:prop4}).
In Section~\ref{sec:phi} these results are extended to the case
where the Krylov approach is employed to approximate
the $\phif$-functions of matrices (generalizing
the exponential function), see Theorem~\ref{theorem:upperboundphi}.
In Section~\ref{sec:saadest},
improved approximations derived from \tjnew{a corrected}
Krylov process~\cite{saad92s} are discussed, and corresponding
error estimators are analyzed,
including an asymptotically correct \tjnew{true} upper bound on the
error (Theorem~\ref{theorem:asymerrorpricorrected}).
This approach can be used to increase the order, but it has the drawback
of violating mass conservation.
In Proposition~\ref{upperboundphi} error estimates are particularized
to the Hermitian case.
Another view on defect-based error estimation is presented
in Section~\ref{sec:defectquad}.

Section~\ref{sec:timeint} is devoted to practical application of the
various error estimators for the control of the time steps $ t $
including smaller substeps~$ \Delta t $ if it appears indicated.
In Section~\ref{sec:num} we present
numerical results \tjnew{for a finite difference discretization of
the free Schr{\"o}dinger equation, a Hubbard model of solar cells,
the heat equation,
and a convection-diffusion problem,}
illustrating our theoretical results. Additional practical
aspects are also investigated:
A priori estimates and the role of restarting are discussed in particular
in the context of practical step-size adaptation.
Finally, we demonstrate the computational efficiency of our adaptive strategy.

%\x{Advantage of our error estimator for skew-Hermitian problem.}
%Although our results apply to general matrices, we focus on the skew-Hermitian
%case pertaining in particular to Schr{\"o}dinger equations. In that case,
%the evolution of (\ref{eq0}) is unitary, and a possible attenuation
%or decay of the solution does not have to be taken into account,
%neither in the theoretical bounds nor for practical considerations
%of error and step-size control.

\section{Problem setting, Krylov approximation, and defect-based representation of the approximation error}
\label{sec:problem}

%\wa{Nonexpansive - also comment on general case where $ \mu_2(A) $
%is known.}

\tjnew{We discuss the approximation of the matrix exponential,
\begin{equation}\label{exp(itA)v}
E(t)v = \ee^{\sca\,t \Opgen} v, \quad \Opgen \in \C^{n\times n},~~\sca\in\C,
\end{equation}
with step size $ t $, applied to an initial vector $ v \in \C^n $.
To simplify the notation we assume $ |\sca|=1 $ and $ \|v\|_2=1 $ without loss of generality.
In many relevant applications (Schr{\"o}dinger-type problems) a complex prefactor is applied to the matrix $A$.
The parameter $\sca$ is introduced here to separate the prefactor of the
matrix $A$.
The standard notation for Schr{\"o}dinger-type problems
is obtained in~\eqref{exp(itA)v} with $\sca=-\ii$ and a Hermitian matrix $A$.
For such problems our notation is helpful to simplify the construction
of the Krylov subspace.
%In many relevant applications the matrix $A$ is 
%Hermitian.
%The parameter $\sca$ is introduced here to separate the phase of the
%matrix $A$ to benefit from cheaper construction of the Krylov subspace. 
%For a Hermitian matrix $A$,
%$\sigma=1$ leads to the Hermitian case and $\sigma =-\ii$ to the skew-Hermitian case.

The exponential $ E(t)=\ee^{\sca\,t \Opgen} $ satisfies the
matrix differential equation
\begin{equation*}
 E'(t) = \sca\,\Opgen E(t),\quad E(0)=I.
\end{equation*}

We assume that $\mu_2(\sigma A)\leq 0$, where $\mu_2(\sigma A)$ denotes the logarithmic norm of $\sigma A$,
or equivalently, $W(\sigma A)\subseteq \C_-$, where $W(\sigma A)$ denotes the field of values of $\sigma A$
and we will refer to this assumption as the \emph{nonexpansive case}.
$\mu_2(\sigma A)\leq 0$ implies $\|E(t)\|_2\leq 1$ for $t\geq 0$.
This is essentially a technical assumption, and most of our theoretical results carry over to a more general setting,
in particular if a~priori information about $\mu_2(\sigma A)$ is available, such that
$E(t)$ can be estimated as $\|E(t)\|_2  \leq \ee^{\mu_2(\sigma A)}$.

For the skew-Hermitian case
with $ \sca = -\ii $ we
write\footnote{In this case the matrix $ A $ is usually named $ H $ (Hamiltonian).}
\begin{equation*}
E(t)v = \ee^{-\ii\,t H} v, \quad H \in \C^{n\times n}~~\text{Hermitian.}
\end{equation*}
In this case, $ E(t) $ represents a unitary evolution,
i.e., $ \|E(t)\|_2 = 1$.
}

\paragraph{Krylov subspaces and associated identities.}
The numerical approximation of~\eqref{exp(itA)v} considered here
(see~\eqref{Smv} below)
is based on the conventional Krylov subspace
\begin{equation*}
\Kry_m(\Opgen,v) = \vecspan\{v,\Opgen v,\ldots,\Opgen^{m-1}v\} \subseteq \C^n.
\end{equation*}
First, \tjnew{an orthonormal} basis of $ \Kry_m(\Opgen,v) $ is obtained by the
well-known Arnoldi iteration, see~\cite{saad03}.
This produces a basis matrix $ V_m \in \C^{n \times m} $
satisfying $ V_m^\ast\,V_m = I_{m \times m} $,
and an upper Hessenberg matrix $ T_m \in \C^{m\times m} $
such that the Krylov identity\footnote{
          Here, $ {e_m = (0,\ldots,0,1)^\ast \in \C^m} $,
          and in the sequel we also denote
          $ {e_1 = (1,0,\ldots,0)^\ast \in \C^m} $.}
\begin{equation}\label{krylovidentity}
\Opgen\,V_m = V_m T_m + \tau_{m+1,m}\,v_{m+1}e_m^\ast
\end{equation}
is valid, with $\tau_{m+1,m}\in\R_+$
and $v_{m+1}\in\C^n$ with $\|v_{m+1}\|_2=1$.

\begin{remark}
We are assuming that the Arnoldi iteration is executed until the
desired dimension~$ m $. Then, by construction, all lower
diagonal entries of $ T_m $ are positive~\cite{saad03}.
If this is not the case, i.e., if a breakdown occurs, it is known
that this breakdown is {\em lucky,}, i.e., the approximation~\eqref{Smv}
below obtained in the step before breakdown is already exact,
see~\cite{saad92s}.
\end{remark}

For the case of a Hermitian matrix $A$
the Krylov subspace can be constructed using the Lanczos iteration,
which is a special case of the Arnoldi iteration,
resulting in a tridiagonal matrix $ T_m \in \R^{m \times m} $.
In the following we discuss the general case
and comment on the case of \tjnew{a Hermitian} matrix~$ A $
whenever appropriate.

The following identities hold true due to
the upper Hessenberg [tridiagonal] structure of~$ T_m $ together with~\eqref{krylovidentity}:
\begin{equation} \label{eTe}
e_m^\ast\,T_m^{j}\,e_1 = 0 \quad \text{for}~~ j=0,\ldots,m-2,
\end{equation}
and
\tjnew{
\begin{equation} \label{HVTm}
\Opgen^{j}v = V_m T_m^{j} e_1,~~~0\leq j \leq m-1,
\end{equation}
see for instance~\cite[Theorem 2]{DK89} or~\cite{saad92s}.}
Furthermore, let
\begin{equation} \label{gammam}
\gamma_m  = e_m^\ast\,T_m^{m-1} e_1
          = \prod_{j=1}^{m-1} (T_m)_{j+1,j}\,,
\end{equation}
where the claimed identity also follows from the
upper Hessenberg [tridiagonal] structure of~$ T_m $.

\paragraph{Krylov approximation.}
The standard Krylov approximation to $ E(t)v $ is
\begin{equation} \label{Smv}
S_m(t)v  = V_m\,\ee^{\sca\,t T_m}\,V_m^\ast\,v
         = V_m\,\ee^{\sca\,t T_m}e_1.
\end{equation}
We denote the corresponding error operator by $ L_m(t) $, with
\begin{equation} \label{Lerror}
L_m(t)=E(t)-S_m(t) \in \C^{n \times n}.
\end{equation}

\paragraph{Defect-based integral representation of the approximation error.}
We define the \emph{defect}\, operator $ D_m(t) $ of $ S_m(t) $ by
\begin{equation*}
D_m(t) = \sca\,\Opgen\,S_m(t) - S_m'(t) \in \C^{n \times n}.
\end{equation*}
Then, $ L_m(t)v $ and $ D_m(t)v $ are related via the differential equation
\begin{equation*}
L_m'(t)v = \sca\,\Opgen\,L_m(t)v + D_m(t)v, \quad L_m(0)v=0,
\end{equation*}
whence
\begin{equation} \label{LDint}
L_m(t)v = \int_0^t E(t-s)\,D_m(s)v\,\dd s.
\end{equation}
An explicit representation for $ D_m(s)v $
is obtained from~\eqref{krylovidentity},
\begin{align}
D_m(s)v
&= \sca\,\Opgen\,V_m\,\ee^{\sca\,s\,T_m}e_1 - \sca\,V_m T_m\,\ee^{\sca\,s\,T_m}e_1
 = \sca\,(\Opgen\,V_m - V_m T_m)\,\ee^{\sca\,s\,T_m}e_1 \notag \\
&= \sca \tau_{m+1,m}\,\big(e_m^\ast\,\ee^{\sca\,s\,T_m}e_1\big)\,v_{m+1}.
\label{Dm(s)v}
\end{align}
Asymptotically for $ t \to 0 $,
\begin{equation} \label{Dm-lead}
D_m(t)v = \sca \tau_{m+1,m}\,\gamma_m\frac{(\sigma t)^{m-1}}{(m-1)!}\,v_{m+1}
+ {\mathcal{O}}(t^m),
\end{equation}
which follows from the Taylor series representation
for $ \ee^{\sca\,t\,T_m} $ together with~\eqref{eTe} and~\eqref{gammam}.
\tjnew{Thus, by~\eqref{LDint} and~\eqref{Dm-lead} we obtain
\begin{equation}\label{orderofDmLm}
\| D_m(t)v \| = {\mathcal{O}}(t^{m-1}),
\quad \text{and} \quad
\| L_m(t)v \| = {\mathcal{O}}(t^{m}).
\end{equation}}
We can also characterize the asymptotically leading term of the error:

\begin{proposition}\label{prop:prop3}
\noindent
For any $ A \in \C^{\tjnew{n \times n}} $
the error $ L_m(t)v $ satisfies the asymptotic relation
\begin{equation} \label{asy-leading}
L_m(t)v = \tau_{m+1,m}\gamma_m\,\frac{(\sca\,t)^{m}}{m!}v_{m+1} + R_{m+1}(t),
\quad R_{m+1}(t) = {\mathcal{O}}(t^{m+1}),
\end{equation}
for $ t \to 0 $.
\end{proposition}
\tjnew{
\begin{proof}
Taylor expansion. Due to $ L_m(t)v = {\mathcal{O}}(t^m) $, \tjnew{see~\eqref{orderofDmLm},}
\begin{equation}\label{prop-3-proof}
\begin{aligned}
&L_m(t)v = E(t)v - S_m(t)v
= \frac{(\sca\,t)^m}{{m!}}(\Opgen^m v - V_m T_m^m\,e_1) + R_{m+1}(t),\\
&\text{with Taylor remainder}~~ R_{m+1}(t) = {\mathcal{O}}(t^{m+1}).
\end{aligned}
\end{equation}
Multiplying the
identity~\tjnew{{\eqref{HVTm}} (with $ j = m-1 $) by $ A $} and using~\eqref{krylovidentity} gives
\begin{align*}
\Opgen^m\,v
&= \Opgen\,V_m T_m^{m-1} e_1
 = (V_m T_m + \tau_{m+1,m}\,v_{m+1} e_m^\ast) T_m^{m-1} e_1 \\
&= V_m T_m^m\,e_1 + \tau_{m+1,m}\, (e_m^\ast\,T_m^{m-1} e_1)\,v_{m+1}
 = V_m T_m^m\,e_1 + \tau_{m+1,m} \gamma_m v_{m+1},
\end{align*}
whence~\eqref{prop-3-proof} simplifies to~\eqref{asy-leading}. \qed
\end{proof}
}

\tjnew{
\begin{remark}
The Taylor remainder $ R_{m+1} $ in~\eqref{asy-leading}
can be specified in a more explicit way showing its dependence on~$ m $,
$$
R_{m+1}(t)
= \frac{(\sigma\,t)^{m+1}}{m!} \int_{0}^{1} \big(A^{m+1} \ee^{\sigma\,\theta\,t\,A} v - V_m T_m^{m+1}\ee^{\sigma\,\theta\,t\,T_m} e_1 \big){(1-\theta)}^m\,\dd\theta .
$$
\end{remark}}

\section{An upper error bound \tjnew{for the \tjnew{nonexpansive case} in~\eqref{exp(itA)v}}}\label{sec:theorypart}
For the \tjnew{nonexpansive case} we have $\| E(t-s) \|_2 \leq 1$ for $0\leq s \leq t$,
and~\eqref{LDint} implies
\begin{equation*}
\|L_m(t)v\|_2
= \Big\| \int_0^t E(t-s)\,D_m(s)v\,\dd s\ \Big\|_2
\leq \int_0^t \|D_m(s)v\|_2\,\dd s.
\end{equation*}
With $\|v_{m+1}\|_2 = 1$, and
\begin{subequations} \label{errestintdefect}
\begin{equation}\label{defdelta}
\delta_m(s) = e_m^\ast\,\ee^{\sca\,s\,T_m}e_1
            = \big( \ee^{\sca\,s\,T_m} \big)_{m,1},
\end{equation}
together with~\eqref{Dm(s)v} we obtain
\begin{equation}\label{errorapprox1}
\|L_m(t)v\|_2 \leq \tau_{m+1,m} \int_0^t |\delta_m(s)| \, \dd s.
\end{equation}
\end{subequations}
This estimate is also given in~\cite[Section~III.2]{lubich08}.
\tjnew{Of course, the integral in~\eqref{errorapprox1} cannot be computed exactly.}
In~\cite{lubich08} it is proposed to use numerical quadrature\footnote{See also Section~\ref{sec:defectquad} below.}
to approximate the integral in~\eqref{errorapprox1}.
\tjnew{In contrast, our aim here is to derive a computable upper bound.}
We proceed in two steps.\footnote{{In the sequel, the argument
                                  of $ \delta_m(\cdot) $ is again denoted
                                  by $ t $ instead of $ s $.}}

\paragraph{Analytic matrix function via interpolation on the spectrum.}
To approximate the error integral in~\eqref{errorapprox1}
we use the representation of matrix exponentials via Hermite
interpolation of the scalar exponential function on the spectrum
of the matrix~$ T_m $, see~\cite[Chap.~1]{higham08}:
If $\mu_1,\ldots,\mu_r$ ($ r \leq m $)
denote the distinct eigenvalues of~$ T_m $
and $ n_j $ is the dimension of the largest Jordan block associated
with~$ \mu_j $, then
\begin{equation}\label{pt}
\ee^{\sca\,t T_m} = p_t(T_m),
\end{equation}
where $ p_t(\lambda) $ is the {Hermite}
interpolant of degree $ \leq m-1 $ of
the function
\begin{equation}\label{def:thetadd}
{f_t}(\lambda) = \ee^{\sca\,t\lambda}
\end{equation}
over the nodes $ \mu_1,\ldots,\mu_m $ in the sense of~\cite[(1.7)]{higham08},
\begin{equation*}
p_t^{(\ell)}(\mu_j) = f_t^{(\ell)}(\mu_j),
\quad j=1,\ldots,r, \quad \ell=0,\ldots,n_j-1.
\end{equation*}
\wa{For a general matrix, the degree of $ p_t $ may be smaller than
$ m-1 $. However, in our context a special case occurs: Since
the lower diagonal entries of $ T_m $ do not vanish, $ T_m $ is
nonderogatory, i.e., for each eigenvalue $ \mu_j $ the associated
eigenspace is one-dimensional, see~\cite[Section~3.1]{horn85}.
Then, $ \sum_{j=1}^{r} n_j = m $, which implies that the degree of
$ p_t $ is exactly $ m-1 $.}

\wa{In the following we denote the full sequence of the $ m $
eigenvalues of $ T_m $ by $ \lambda_1,\ldots,\lambda_m $.}
By applying basic properties of the Krylov decomposition and
imposed conditions on the numerical range
of $\Opgen$ we obtain
\begin{equation}\label{subsetofspectrum}
\text{spec}(\sca\,T_m) \subseteq
W(\sca\,T_m) \subseteq
W(\sca\,\Opgen) \subseteq \C_-.
\end{equation}

The following proposition is partially related to~\tjnew{\cite[Sec.~3]{CeMo97}} or~\cite{chinesen17}.
\wa{Here, divided differences have to be understood
in the general sense, i.e., in the confluent sense if multiple
eigenvalues occur; for the detailed definition and properties
see~\cite[Section~B.16]{higham08}.}

\begin{proposition}%[Representation and estimate of $\delta_m(t)$]
                    %cf.~\text{\rm\cite[Chap.~3]{CeMo97}}]
                    \label{prop:prop1}

\noindent
Let $T_m \in \C^{m\times m}$ be an upper Hessenberg matrix with eigenvalues $ \lambda_1,\ldots,\lambda_m $
and $ \text{spec}(\sca T_m) \subseteq \C_- $.
Then the function $ \delta_m(t) $ defined as in~\eqref{defdelta}, i.e.,
\begin{equation*}
\delta_m(t) = e_m^\ast\,\ee^{\sca\,t\,T_m}e_1
            = \big( \ee^{\sca\,t\,T_m} \big)_{m,1},
\end{equation*}
satisfies
\begin{equation}\label{deltam-repr}
\delta_m(t) =
\divdif{\lambda_1,\ldots,\lambda_m}{f_t} \gamma_m
\leq {\frac{t^{m-1}}{(m-1)!}\,\gamma_m,}
\end{equation}
with $ \gamma_m $ from~\eqref{gammam} and
where $\divdif{\lambda_1,\ldots,\lambda_m}{f_t}$ is
the {$ (m-1) $-th} divided difference over
$ \text{\rm spec}(T_m) $
of the function $ f_t $ defined in~\eqref{def:thetadd}.
\end{proposition}

\begin{proof}
We proceed from the Newton representation of the
interpolant $p_t(\lambda) $ from~\eqref{pt},
\begin{equation*}
p_t(\lambda) = \sum_{j=0}^{m-1}
               \divdif{\lambda_1,\ldots,\lambda_{j+1}}{f_t}\,\omega_j(\lambda),
\end{equation*}
with
$ \omega_j(\lambda) = (\lambda-\lambda_1) \cdots (\lambda-\lambda_{j}) $.
%for $ j=0,\ldots,m-1 $.
{From~\eqref{eTe} and by definition of $ \gamma_m $, % (monic)
it is obvious that the $ \omega_j $ satisfy}
\begin{equation*}
e_m^\ast\,\omega_j(T_m)\,e_1 =
\left\{\begin{array}{ll}
0,       & j=0,\ldots,m-2, \\
\gamma_m,& j=m-1.
\end{array}\right.
\end{equation*}
Together with~\eqref{pt} this shows that the identity claimed
in~\eqref{deltam-repr} is valid:
\begin{equation*}
\delta_m(t)
= e_m^\ast\,\ee^{\sca\,t T_m}\,e_1
= e_m^\ast\,p_t(T_m) e_1
= \sum_{j=0}^{m-1} \divdif{\lambda_1,\ldots,\lambda_{j+1}}{f_t} e_m^\ast\,\omega_j(T_m)e_1
= \divdif{\lambda_1,\ldots,\lambda_m}{f_t} \gamma_m.
\end{equation*}
According to~\cite[(B.28)]{higham08} the divided
difference can be estimated by
\begin{align*}
&|f_t[\lambda_1,\ldots,\lambda_m]| \leq
 \frac{\max_{z \in \Omega} D^{(m-1)}f_t(z)}{(m-1)!}\\
&\text{for convex $\Omega \subseteq \C$
which contains all eigenvalues $ \lambda_j$.}
\end{align*}
With $D^{(m-1)}f_t(\lambda) =
(\sca\,t)^{m-1} \ee^{\sca\,t\lambda} $,
$ |\sca|=1 $ and $\text{Re}(\sca \lambda_j)\leq 0$ we obtain
\begin{equation*}
|\divdif{\lambda_1,\ldots,\lambda_m}{f_t}| \leq \frac{t^{m-1}}{(m-1)!},
\end{equation*}
which implies the estimate~\eqref{deltam-repr}
for $ \delta_m(t) $. \qed
\end{proof}

\paragraph{Error estimate and asymptotical correctness.}

Now we apply Proposition~\ref{prop:prop1} in the context of
our Krylov approximation.

\begin{theorem}[Computable upper bound]\label{theorem:upperboundexp}

\noindent
For the nonexpansive case the error $ L_m(t)v $ of the Krylov approximation~\eqref{Smv}
to $ E(t)v $ satisfies
\begin{equation} \label{asyerrest}
\| L_m(t)v \|_2 \leq \tau_{m+1,m}\gamma_m\,\frac{t^{m}}{m!}
\end{equation}
with $ \tau_{m+1,m} $ from~\eqref{krylovidentity} and
$ \gamma_m $ from~\eqref{gammam}.
\end{theorem}
\begin{proof}
% we use nonexpansive and~\eqref{subsetofspectrum} for assumptions of Proposition~\ref{prop:prop1}.
We proceed from~\eqref{errestintdefect}.
For $ \delta_m $ defined in~\eqref{defdelta},
Proposition~\ref{prop:prop1} implies
\begin{equation*}
|\delta_m(s)|
%= |e_m^\ast\,\ee^{-\ii\,s T_m} e_1|
%|\psi_{m-1}(t)| \gamma_m =
%|\divdif{\lambda_1,\ldots,\lambda_m}{f_t}| \gamma_m
\leq \frac{s^{m-1}}{(m-1)!} \gamma_m,
\end{equation*}
and this gives an upper bound for the error
integral~\eqref{errorapprox1}:
\begin{equation*}
\|L_m(t)v\|_2 \leq
%\tau_{m+1,m}\int_0^t |e_m^\ast\,\ee^{\ii\,s\,T_m} e_1|\,\dd s \leq
\tau_{m+1,m} \gamma_m \int_0^t \frac{s^{m-1}}{(m-1)!}\,\dd s
= \tau_{m+1,m}\gamma_m \frac{t^{m}}{m!},
\end{equation*}
which completes the proof. 
\tjnew{
Proposition~\ref{prop:prop1} is applied here in the nonexpansive case ($ W(\sigma A) \subseteq \C_- $)
which implies the requirement $ \text{spec}(\sca T_m) \subseteq \C_- $, see~\eqref{subsetofspectrum}.}

\qed
\end{proof}

The upper bound~\eqref{asyerrest} corresponds to the 2-norm of the
leading error term~\eqref{asy-leading} according to Proposition~\ref{prop:prop3}.
It is easily computable
from the \tjnew{Krylov} decomposition~\eqref{krylovidentity}.
We denote the error estimate given by~\eqref{asyerrest} as
\begin{equation}\label{errornew}\tag{$\Era$}
\Era = \tau_{m+1,m}\gamma_m\,\frac{t^{m}}{m!}.
\end{equation}
%We will refer to $ \Era $ as the \tjnew{strict} error estimate,
%\tjnew{where 'strict' means that it constitutes a true proven upper bound.}
\begin{proposition}[Asymptotical correctness]\label{prop:prop4}

\noindent
The upper bound~\eqref{asyerrest} is asymptotically correct
for $ t \to 0 $, i.e.,
\begin{equation} \label{asyerrcorr}
\| L_m(t)v \|_2 = \tau_{m+1,m}\gamma_m\,\frac{t^m}{m!} + {\mathcal{O}}(t^{m+1}).
\end{equation}
\end{proposition}

\begin{proof}
The asymptotic estimate
\begin{align*}
\bigg| \| L_m(t)v \|_2 - \tau_{m+1,m}\gamma_m\,\frac{t^m}{m!} \bigg|
&=
\bigg| \| L_m(t)v \|_2 -
        \Big\| \tau_{m+1,m}\gamma_m\,\frac{(\sca\,t)^{m}}{m!}v_{m+1} \Big\|_2
\bigg| \\
&\leq
\Big\| L_m(t)v - \tau_{m+1,m}\gamma_m\,\frac{(\sca\,t)^{m}}{m!}v_{m+1} \Big\|_2
= {\mathcal{O}}(t^{m+1})
\end{align*}
is valid due to~Proposition~\ref{prop:prop3},
and this proves~\eqref{asyerrcorr}. \qed
\end{proof}

%%%%%
%%%%%
%%%%%
\begin{remark}\label{SIkrylov}
In \cite[Section 4]{eshhoc06} a defect-based error formulation is given
for the shift-and-invert Krylov approximation of the matrix exponential function.
In contrast to the standard Krylov method, the defect is not of order $m-1$ for $t\to 0$ there.
Hence, our new results do not directly apply to shift-and-invert Krylov approximations.
A study of a posteriori error estimates for the shift-and-invert approach is a topic of future investigations.
\end{remark}

\section{Krylov approximation to $\phif$-functions.} \label{sec:phi}
As another application we consider the so-called $ \phif $-functions,
with power series representation
\begin{subequations}\label{defphi}
\begin{equation}\label{defphiseries}
\phif_p(z) = \sum_{k=0}^{\infty} \frac{z^k}{(k+p)!}, \quad p\geq 0.
\end{equation}
We have $ \phif_0(z) = \ee^z $, and
\begin{equation} \label{defphiint}
\phif_p(z) = \frac{1}{(p-1)!} \int_0^1 (1-\theta)^{p-1}\ee^{\theta z}\,\dd\theta,
\quad p \geq 1.
\end{equation}
\end{subequations}

As the matrix exponential, $\phif$-functions of matrices also appear in a wide range of applications, such as exponential integrators,
see for instance~\cite{MH11,hochbrucketal98,hocost10,niewri12,sidje98}.
Krylov approximation is a common technique to evaluate $\phif$-functions of matrices applied to a starting vector,
\begin{equation} \label{kry-phi}
\phif_p(\sca\,t \Opgen)v \approx V_m \phif_p(\sca\,t T_m) e_1,~~~~p\geq 0.
\end{equation}
Since $\phif$-functions are closely related to the matrix exponential,
our ideas can be applied to these as well.
We use the following notation for the error in the $\phif$-functions:
\begin{equation}\label{generalizeerrornotation}
L^p_m(t)v = \phif_p(\sca\,t\,\Opgen)v - V_m \phif_p(\sca\,t\,T_m) e_1.
\end{equation}
With~\eqref{generalizeerrornotation} we generalize the previously used notation: $ L_m(t) = L^0_m(t) $.
\begin{theorem}\label{theorem:upperboundphi}
The error of the Krylov approximation~\eqref{kry-phi}
to $ \phif_p(\sca\,t \Opgen)v $ with $ p \geq 0 $ satisfies
\begin{subequations}
\begin{equation}\label{leadingtermphij}
L^p_m(t)v = \tau_{m+1,m}\,\gamma_m \frac{(\sca\, t)^m}{(m+p)!} \, v_{m+1} + \mathcal{O}(t^{m+1}).
\end{equation}
Furthermore, \tjnew{in the nonexpansive case} its norm is bounded by
\begin{equation}\label{eresphij}
\|L^p_m(t)v \|_2  \leq   \tau_{m+1,m}\,\gamma_m\,\frac{t^{m}}{(m+p)!},
\end{equation}
\end{subequations}
and this bound is asymptotically correct for $ t \to 0 $.
\end{theorem}
\begin{proof}
For $p=0$ the result directly follows from Propositions~\ref{prop:prop3},~\ref{prop:prop4} and Theorem~\ref{theorem:upperboundexp}. We now assume $ p \geq 1 $.
Via the series representation~\eqref{defphiseries} of $\phif_p$
we can determine the leading term of the error in an analogous way
as in Proposition~\ref{prop:prop3}:
\begin{align*}
\phif_p(\sca\,t \Opgen)v -  V_m \phif_p(\sca\,t T_m) e_1
&= \frac{(\sca\,t)^m \,(A^m v - V_m T_m^m e_1)}{(m+p)!} + \mathcal{O}(t^{m+1}) \\
%&= \sum_{k=m}^{\infty} \frac{(\sca\,t)^k \,(A^k v - V_m T_m^k e_1)}{(k+p)!} \\
&= \tau_{m+1,m}\,\gamma_m \frac{(\sca\, t)^m}{(m+p)!} \, v_{m+1} + \mathcal{O}(t^{m+1}),
\end{align*}
which proves~\eqref{leadingtermphij}.

Furthermore, proceeding from~\eqref{defphiint} we obtain
\begin{align*}
\phif_p(\sca\,t \Opgen)v - V_m \phif_p(\sca\,t T_m) e_1
&= \frac{1}{(p-1)!} \int_0^1 (1-\theta)^{p-1}\big( \ee^{\sca\,\theta\,t \Opgen} v - V_m\,\ee^{\sca\,\theta\,t T_m} e_1 \big) \,\dd \theta\\
&= \frac{1}{(p-1)!} \int_0^1 (1-\theta)^{p-1} L_m(\theta\,t)v \,\dd \theta,
\end{align*}
with the error $L_m(t)v$ for the matrix exponential case.
Now we apply Theorem~\ref{theorem:upperboundexp} to obtain
\begin{align*}
\| \phif_p(\sca\,t \Opgen)v - V_m \phif_p(\sca\,t T_m) e_1 \|_2
&\leq \frac{1}{(p-1)!} \int_0^1 (1-\theta)^{p-1}  \|L_m(\theta\,t)v\|_2 \,\dd \theta\\
&\leq \tau_{m+1,m}\,\gamma_m\,\frac{t^m}{(p-1)!\,m!} \int_0^1 (1-\theta)^{p-1} \theta^m \,\dd \theta\\
&= \tau_{m+1,m}\,\gamma_m\,\frac{t^{m}}{(m+p)!},
\end{align*}
which proves~\eqref{eresphij}. \qed
\begin{comment}
{
whereat, the last step can be shown similar to the equivalence of~\eqref{defphiseries} and~\eqref{defphiint}.
Equivalence of~\eqref{defphiseries} and~\eqref{defphiint} can be proven by recursive representation of $\phi$-functions and integration by parts.
TODO: good reference or short proof.
Prove:
\begin{equation}\label{equofintandfact}
\frac{1}{(j-1)!}\int_0^1 (1-s)^{j-1} s^m \,\dd s \stackrel{!}{=} \frac{m!}{(m+j)!}.
\end{equation}
By the equivalence of~\eqref{defphiseries} and~\eqref{defphiint} we get
\begin{equation*}
\frac{1}{(j-1)!}\int_0^1 \Big((1-s)^{j-1} \sum_{k=0}^{\infty} \frac{(zs)^k}{k!} \Big)\,\dd s = \sum_{k=0}^{\infty} \frac{z^k}{(k+j)!}.
\end{equation*}
Since this equality holds for all $z\in\C$ we can compare terms of $k=m$, which refer to $z^m$ terms, on both sides to observe
\begin{equation*}
\frac{1}{(j-1)!}\int_0^1 \Big((1-s)^{j-1} \frac{(zs)^m}{m!} \Big)\,\dd s = \frac{z^m}{(m+j)!}.
\end{equation*}
Choosing $z=1$ leads to
\begin{equation*}
\frac{1}{(j-1)!}\int_0^1 \Big((1-s)^{j-1} \frac{s^m}{m!} \Big)\,\dd s = \frac{1}{(m+j)!},
\end{equation*}
which proves~\eqref{equofintandfact}.
}
\end{comment}
\end{proof}

\section{Corrected Krylov approximation for the exponential and $\phif$-functions.}
\label{sec:saadest}

%Error representation from Saad is still used for error estimates in practice, see~\cite{sidje98} or~\cite{niewri12}.

%The proof of the following error representation
%can be based on~\eqref{LDint},\,\eqref{Dm(s)v}
%and using~\eqref{defphiint2}.

Let us recall the well-known error representation given in~\cite{saad92s}.
\begin{proposition}{{\rm{see~\cite[Theorem 5.1]{saad92s}}}}\label{prop:saad92}
With the $\phif$-functions defined in~\eqref{defphi},
the error~\eqref{Lerror} can be represented in the form
\begin{align}\label{sumofsaad}
L_m(t)v =
 \tau_{m+1,m}\,\sca\,t
  \sum_{j=1}^{\infty} e_m^\ast\phif_{j}(\sca\,tT_m)\,e_1 (\sca\,t \Opgen)^{j-1}v_{m+1}.
\end{align}
\end{proposition}

In~\cite{saad92s} it is stated that, typically,
the first term of the sum given in
Proposition~\ref{prop:saad92}, formula~\eqref{errorsaadterm1},
is already a good approximation to $L_m(t)v$.
Analogously to~\cite[Section 5.2]{saad92s}
we use the notation $\Erone$
for the norm of this term,
\begin{equation}\label{errorsaadterm1}\tag{$\Erone$}
\Erone = \tau_{m+1,m}\,t\,|e_m^\ast\phif_1(\sca\,t T_m)e_1|.
\end{equation}
In~\cite{chinesen17} it is even shown
that~$\Erone$ is an upper bound up
to a factor depending on spectral properties of the matrix $\Opgen$.
For the case of Hermitian $\sca\Opgen$ we show $\|L_m(t)v\|_2 \leq \Erone$ in Proposition~\ref{upperboundphi} below.

In Remark~\ref{rmk:PhiAsymptoticCorrect} below we show that $\Erone$ is also an asymptotically correct
approximation for the error norm
(in the sense of Proposition~\ref{prop:prop4}).
Furthermore, the error estimate $\Erone$ is computable at nearly no extra cost,
see~\cite[Proposition~2.1]{saad92s}.

According to~\cite[Proposition~2.1]{saad92s},
$\phif_1(\sca\,t T_m)$ can be computed from the extended matrix
\begin{subequations}
\begin{equation}\label{definetcor}
\Tcor_m =
\begin{bmatrix}
T_m&0\\
\tau_{m+1,m}\,e_m^\ast&0
\end{bmatrix}\in\C^{(m+1)\times (m+1)}
\end{equation}
as
\begin{equation}\label{computephi1}
\ee^{\sca\,t\Tcor_m} e_1=
\begin{bmatrix}
\ee^{\sca\,tT_m}e_1\\
\tau_{m+1,m}\,\sca\,t \,(e_m^\ast\,\phif_1(\sca\,t T_m)\,e_1)
\end{bmatrix}\in\C^{m+1}.
\end{equation}
\end{subequations}
Equation~\eqref{computephi1} can be used to evaluate the error estimate
$\Erone$
or a corrected Krylov approximation in the form
\begin{equation} \label{Smbar}
\Scor_m(t)v = \Vcor_m\,\ee^{\sca\,t\Tcor_m}e_1
\text{~~~with~~~}
\Vcor_m=\Big[ V_m\,\big|\,v_{m+1} \Big] \in\C^{n\times (m+1)},
\end{equation}
for which the first term of the error expansion according to
Proposition~\ref{prop:saad92} vanishes, see~\cite{saad92s}.
For the error of the corrected Krylov approximation we use the notation
\begin{equation*}
L_m^+(t)v = E(t)v - \Scor_m(t)v.
\end{equation*}

For general $\phif$-functions we obtain an error representation similar to Proposition~\ref{prop:saad92} and
a corrected Krylov approximation for $\phif$-functions.
The corrected Krylov approximation to $\phif_p(\sca\,t\,\Opgen)v$ is given in~\cite[Theorem 2]{sidje98}:
\begin{equation*}
\phif_p(\sca\,t\,\Opgen)v \approx \Vcor_m \phif_p(\sca\,t\,\Tcor_m) e_1
\end{equation*}
with $\Tcor_m$ and $\Vcor_m$ given in~\eqref{definetcor} and~\eqref{Smbar}.
The error of the corrected Krylov approximation is denoted by
\begin{equation}\label{generalizeerrornotation2}
L^{p,+}_m(t)v = \phif_p(\sca\,t\,\Opgen)v - \Vcor_m \phif_p(\sca\,t\,\Tcor_m) e_1.
\end{equation}
\begin{proposition}[\rm{see~\cite[Theorem 2]{sidje98}}]\label{prop:generalSaadsum}
\begin{subequations}
The error of the Krylov approximation $L^p_m(t)v$, see~\eqref{generalizeerrornotation}, satisfies
\begin{equation}\label{generalizeerrorSaadsum1}
L^p_m(t)v = \tau_{m+1,m} \sca\,t \sum_{j=p+1}^{\infty} (e_m^\ast \phif_j(\sca\,t\,T_m) e_1) \,(\sca\,t\,\Opgen)^{j-p-1} v_{m+1}.
\end{equation}
The error of the corrected Krylov approximation $L^{p,+}_m(t)v$, see~\eqref{generalizeerrornotation2}, is given by
\begin{equation}\label{generalizeerrorSaadsum2}
L^{p,+}_m(t)v = \tau_{m+1,m} \sca\,t \sum_{j=p+2}^{\infty} (e_m^\ast \phif_j(\sca\,t\,T_m) e_1) \,(\sca\,t\,\Opgen)^{j-p-1} v_{m+1}.
\end{equation}
\end{subequations}
\end{proposition}

The following remark will be used later on.
\begin{remark}\label{rmk:PhiAsymptoticCorrect}
From the representation~\eqref{defphiseries} for the $\phif_j$
together with~\eqref{eTe} and~\eqref{gammam} we observe
\begin{equation} \label{em-phi-e1}
\begin{aligned}
e_m^\ast \phif_j(\sca t T_m) e_1 &=
%\sum_{k=m-1}^{\infty} \frac{(\sca t)^k e_m^\ast T_m^k\,e_1}{(k+j)!}
\frac{(\sca t)^{m-1} e_m^\ast T_m^{m-1}\,e_1}{(m-1+j)!}
+ \mathcal{O}(t^m)
= \gamma_m \,\frac{(\sca t)^{m-1}}{(m-1+j)!}
+ \mathcal{O}(t^m).
\end{aligned}
\end{equation}
By~\eqref{em-phi-e1} we observe $ e_m^\ast\phif_{j}(\sca\,tT_m)\,e_1 = \mathcal{O}(t^{m-1}) $ for $ j \geq 0 $
and we conclude that the asymptotically leading order term of $ L^p_m(t)v $ for $ t \to 0 $ is obtained by the leading term ($j=p+1$) of the series~\eqref{generalizeerrorSaadsum1}:
\begin{subequations}
\begin{equation}\label{Lpleadingtermsum}
L^p_m(t)v = \tau_{m+1,m} \sca\,t (e_m^\ast \phif_{p+1}(\sca\,t\,T_m) e_1 ) v_{m+1} + \mathcal{O}(t^m).
\end{equation}
Analogously we obtain the asymptotically leading order term of $ L^{p,+}_m(t)v $ for $ t \to 0 $ by the leading term ($j=p+2$) of the series~\eqref{generalizeerrorSaadsum2}:
\begin{equation}\label{Lpplusleadingtermsum}
L^{p,+}_m(t)v = \tau_{m+1,m} (\sca\,t)^2 (e_m^\ast \phif_{p+2}(\sca\,t\,T_m) e_1) \,\Opgen v_{m+1} + \mathcal{O}(t^{m+1}).
\end{equation}
\end{subequations}
The asymptotically leading terms in~\eqref{Lpleadingtermsum} and~\eqref{Lpplusleadingtermsum} can be used as error estimators:
\begin{subequations}
\begin{equation}\label{LPnextphierrorest}
\|L^p_m(t)v\|_2 \approx \tau_{m+1,m}\,t |e_m^\ast \phif_{p+1}(\sca\,t\,T_m) e_1|
\end{equation}
and
\begin{equation}\label{LPplusnextphierrorest}
\|L^{p,+}_m(t)v\|_2 \approx \|\Opgen v_{m+1}\|_2 \tau_{m+1,m}\,t^2 |e_m^\ast \phif_{p+2}(\sca\,t\,T_m) e_1|.
\end{equation}
\end{subequations}
The error estimators~\eqref{LPnextphierrorest} and~\eqref{LPplusnextphierrorest} are already suggested in \cite{sidje98,niewri12}.
We will refer to them as $\Erone$ in the context of the $\phif$-functions with standard and corrected Krylov approximation,
generalizing the corresponding quantities for the exponential case $p=0$.
\end{remark}

We also obtain \tjnew{true} upper bounds for the matrix exponential
($p=0$) and general $\phif$-functions with $p \geq 1$.

\begin{theorem}\label{theorem:asymerrorpricorrected}
The error of the corrected Krylov approximation~\eqref{generalizeerrornotation2}
to $ \phif_p(\sca\,t \Opgen)v $ with $ p \geq 0 $ satisfies
\begin{subequations}
\begin{equation}\label{leadingtermphijcor}
L^{p,+}_m(t)v  = \tau_{m+1,m}\,\gamma_m \frac{(\sca\, t)^{m+1}}{(m+p+1)!} \, \Opgen v_{m+1} + \mathcal{O}(t^{m+2}).
\end{equation}
Furthermore, \tjnew{in the nonexpansive case} its norm is bounded by
\begin{equation}\label{eresphijcor}
\|L^{p,+}_m(t)v\|_2 \leq    \|\Opgen v_{m+1}\|_2\, \tau_{m+1,m}\,\gamma_m \frac{t^{m+1}}{(m+p+1)!},
\end{equation}
\end{subequations}
and this bound is asymptotically correct for $ t \to 0 $.
\end{theorem}

\begin{proof}
Applying~\eqref{em-phi-e1} (with $j=p+2$) to~\eqref{Lpplusleadingtermsum} shows~\eqref{leadingtermphijcor}:
\begin{equation*}
L^{p,+}_m(t)v = \tau_{m+1,m}\, \gamma_m\, \frac{(\sca\,t)^{m+1}}{(m+p+1)!} \,\Opgen v_{m+1} + \mathcal{O}(t^{m+2}).
\end{equation*}
%{Proof directly holds for the case $p=0$, because we use inequalities for error $\phi_{p+1}$, error estimate is only applied to $\phi_j$ function with $j\geq 1$.}
From Proposition~\ref{prop:generalSaadsum} we observe
\begin{align*}
L^{p,+}_m(t)v = \sigma\,t\,\Opgen \,\,L^{p+1}_m(t)v.
\end{align*}
Using the integral representation analogously as in the proof of Theorem~\ref{theorem:upperboundphi} for $L^{p+1}_m(t)v$ and
formula~\eqref{LDint} for $L_m(t)v$, we obtain
\begin{align*}
L^{p,+}_m(t)v &= \sigma\,t\,\Opgen\,\, L^{p+1}_m(t)v = \tau_{m+1,m} \sigma\,t \frac{1}{p!} \int_0^1 (1-\theta)^{p} \Opgen\, L_m(\theta\,t)v \,\dd \theta\\
&= \tau_{m+1,m} \sigma\,t \frac{1}{p!} \int_0^1 (1-\theta)^{p} \int_0^{\theta t} \ee^{\sca (\theta t-s) \Opgen} \Opgen v_{m+1} \delta_m( s)\,\dd s \,\dd \theta.
\end{align*}
With norm inequalities \tjnew{(note the nonexpansive case)} and Proposition~\ref{prop:prop1} we obtain
\begin{align*}
\|L^{p,+}_m(t)v\|_2 &\leq \tau_{m+1,m} \,t\,\|\Opgen v_{m+1}\|_2 \frac{1}{p!} \int_0^1 (1-\theta)^{p} \int_0^{\theta t}  |\delta_m( s)|\,\dd s \,\dd \theta\\
&\leq \|\Opgen v_{m+1}\|_2 \tau_{m+1,m} \gamma_m\, t^{m+1}\,\frac{1}{p!\,m!} \int_0^1 (1-\theta)^{p} \theta ^m\,\dd s \,\dd \theta\\
&= \|\Opgen v_{m+1}\|_2  \tau_{m+1,m} \gamma_m \frac{t^{m+1}}{(m+p+1)!},
\end{align*}
which proves~\eqref{eresphijcor}. 
\tjnew{Proposition~\ref{prop:prop1} is applied here in the nonexpansive case,
see also the proof of Theorem \ref{theorem:upperboundexp}.}\hfill\qed
\end{proof}

If the error estimate~\eqref{eresphijcor}
is to be evaluated, the effort of the computation of $\|\Opgen v_{m+1}\|_2$
is comparable to one additional step of the \tjnew{Krylov} iteration.

As mentioned before, we also can show that for Hermitian $\sca\Opgen$
the estimate $\Erone$ gives a \tjnew{true} upper bound:

\begin{proposition}\label{upperboundphi}
\tjnew{For the nonexpansive case} with $\sca=1$ and \tjnew{a Hermitian} matrix $\Opgen$ we obtain
\begin{equation*}
|\delta_m(t)| = \delta_m(t) > 0 ~~~\text{for}~~t > 0.
\end{equation*}
This leads to the following upper bounds for the errors $L^p_m$ and $L^{p,+}_m$ with $p\geq 0$:
\begin{subequations}
\begin{equation}\label{hermitianphiestimate1}
\|L^p_m(t)v\|_2 \leq \tau_{m+1,m} \, t \, \underbrace{e_m^\ast \, \phif_{p+1}(tT_m) \, e_1}_{\geq\, 0}
\end{equation}
and
\begin{equation}\label{hermitianphiestimate2}
\|L^{p,+}_m(t)v\|_2 \leq \|\Opgen v_{m+1}\|_2 \,\tau_{m+1,m} \, t^2 \, \underbrace{e_m^\ast \, \phif_{p+2}(tT_m) \, e_1}_{\geq\, 0}.
\end{equation}
\end{subequations}
\end{proposition}
\begin{proof}
For \tjnew{a Hermitian} matrix $\Opgen$ we obtain a symmetric, tridiagonal matrix $T_m$ with distinct, real eigenvalues via Lanczos approximation, see~\cite[Chap. 3.1]{horn85}.
By Proposition~\ref{prop:prop1} we observe
\begin{equation*}
\delta_m(t) = \divdif{\lambda_1,\ldots,\lambda_m}{f_t} \gamma_m
\end{equation*}
with $f_t(\lambda)=\ee^{t\,\lambda}$ for the case $\sca=1$.
For divided differences of real-valued functions over real nodes we obtain $\divdif{\lambda_1,\ldots,\lambda_m}{f_t} \in \R$ and
\begin{equation}\label{ddreal}
\divdif{\lambda_1,\ldots,\lambda_m}{f_t} = \frac{D^{(m-1)} f_t(\xi)}{(m-1)!} = \frac{t^{m-1} \ee^{t\xi}}{(m-1)!}~~~\text{for}~~\xi\in[\lambda_1,\lambda_m].
\end{equation}
Equation~\eqref{ddreal} shows $\divdif{\lambda_1,\ldots,\lambda_m}{f_t} > 0$ and with $\gamma_m > 0 $ we conclude
\begin{equation*}
\delta_m(t) > 0,~~~\text{and}~~~|\delta_m(t)|=\delta_m(t).
\end{equation*}
We continue with~\eqref{hermitianphiestimate1} in the case $p=0$:
\begin{equation*}
\|L_m(t)v\|_2\leq \tau_{m+1,m} \int_0^t |\delta_m(s)|\,\dd s = \tau_{m+1,m} \int_0^t e_m^\ast\,\ee^{sT_m} \,e_1 \,\dd s = \tau_{m+1,m} t \, e_m^\ast \, \phif_1(tT_m) \, e_1.
\end{equation*}
For the case $p \geq 1$ we start analogously to Theorem~\ref{theorem:upperboundphi}.
Using definition~\eqref{defphiint} for the $\phif$-functions
and resorting to the case $p=0$ we find
\begin{align*}
\| L_m^p(t)v \|_2
&= \frac{1}{(p-1)!} \int_0^1 (1-\theta)^{p-1}  \|L_m(\theta\,t)v\|_2 \,\dd \theta\\
&\leq \tau_{m+1,m}\,t\,\frac{1}{(p-1)!} \int_0^1 (1-\theta)^{p-1} \theta \, e_m^\ast \, \phif_1(\theta\,t\,T_m) \, e_1  \,\dd \theta.
\end{align*}
Evaluation of the integral yields
\begin{equation}\label{partofproofphicor1}
\begin{aligned}
\| L_m^p(t)v \|_2
&\leq \tau_{m+1,m}\,t\,\frac{1}{(p-1)!} \int_0^1 (1-\theta)^{p-1} \theta \, e_m^\ast \, \phif_1(\theta\,t\,T_m) \, e_1  \,\dd \theta\\
&= \tau_{m+1,m}\,t\,\sum_{k=0}^{\infty}\frac{e_m^\ast \,(\,t\,T_m)^k \, e_1}{(p-1)!\,(k+1)!} \int_0^1 (1-\theta)^{p-1} \theta^{k+1} \,     \,\dd \theta \\
&= \tau_{m+1,m}\,t\,\sum_{k=0}^{\infty} \frac{e_m^\ast \,(\,t\,T_m)^k \, e_1}{(p+k+1)!}\\
&= \tau_{m+1,m}\,t\,e_m^\ast \phif_{p+1}(\,t\,T_m) \, e_1.
\end{aligned}
\end{equation}
This shows~\eqref{hermitianphiestimate1}.
To show~\eqref{hermitianphiestimate2} we start analogously to Theorem~\ref{theorem:asymerrorpricorrected}:
\begin{equation*}
\|L^{p,+}_m(t)v\|_2 = \|t\,\Opgen\,L^{p+1}_m(t)v\|_2 \leq \|\Opgen v_{m+1}\|_2\,\tau_{m+1,m} \,t\, \frac{1}{p!} \int_0^1 (1-\theta)^{p} \int_0^{\theta t}  |\delta_m( s)|\,\dd s \,\dd \theta.
\end{equation*}
Using $|\delta_m(s)|=\delta_m(s)$ and evaluating the inner integral by the $\phif_1$ function, we obtain
\begin{equation*}
\|L^{p,+}_m(t)v\|_2 \leq \|\Opgen v_{m+1}\|_2\,\tau_{m+1,m}\,t^2\,\frac{1}{p!} \int_0^1 (1-\theta)^{p} \theta \, e_m^\ast \, \phif_1(\theta\,t\,T_m) \, e_1  \,\dd \theta.
\end{equation*}
Evaluation of the integral analogously to~\eqref{partofproofphicor1},
\begin{equation*}
\|L^{p,+}_m(t)v\|_2 \leq \|\Opgen v_{m+1}\|_2\,\tau_{m+1,m}\,t^2\,e_m^\ast \phif_{p+2}(\,t\,T_m) \, e_1,
\end{equation*}
completes the proof. \hfill\qed
\end{proof}

%%%%%
%%%%%
%%%%%

\section{Defect-based quadrature error estimates revisited}\label{sec:defectquad}
The term on the right-hand side of~\eqref{asy-leading}
is a computable error estimate, which has been investigated more closely
in Section~\ref{sec:theorypart}.
It can also be interpreted in an alternative way.
\wa{To this end we again proceed from the integral representation~\eqref{LDint},
\begin{equation} \label{LDint-1}
L_m(t)v = \int_0^t \underbrace{E(t-s)\,D_m(s)}_{=:\;\Theta_m(s,t)}v\,\dd s.
\end{equation}
Due to $ \| D_m(t)v \| = {\mathcal{O}}(t^{m-1}) $,
\begin{equation*}
\tfrac{\dd^j}{\dd\,s^j}\,D_m(s)v\big|_{s=0} = 0, \quad j=0,\ldots,m-2,
\end{equation*}
and the same is true for the integrand in~\eqref{LDint-1},
\begin{equation*}
\tfrac{\partial^j}{\partial s^j}\,\Theta_m(s,t)v\big|_{s=0} = 0, \quad j=0,\ldots,m-2.
\end{equation*}
Analogously as in~\cite{auzingeretal13a}, this allows
us to approximate~\eqref{LDint-1} by \tjnew{a Hermite quadrature} formula
in the form
\begin{equation} \label{LDint-1-appr}
\int_0^t \Theta_m(s,t)v\,\dd s \approx \frac{t}{m}\,\Theta_m(t,t)v
                                = \frac{t}{m}\,D_m(t)v.
\end{equation}
From~\eqref{Dm-lead},
\begin{equation*}
\frac{t}{m}\,D_m(t)v =
\tau_{m+1,m}\,\gamma_m\frac{(\sigma t)^{m}}{m!}\,v_{m+1}
+ {\mathcal{O}}(t^{m+1}),
\end{equation*}
which is the same as~\eqref{asy-leading}. This means that the
quadrature approximation~\eqref{LDint-1-appr} approximates
the leading error term in an asymptotically correct way.
}
From~\eqref{LDint-1-appr},~\eqref{Dm(s)v} and~\eqref{defdelta} we obtain
\begin{equation}\label{eq:hermitquadnorm}
\|L_m(t)v\|_2 \approx \tau_{m+1,m}\frac{t}{m}|\delta_m(t)|.
\end{equation}

\wa{
The quadrature error in~\eqref{LDint-1-appr}
is $ {\mathcal{O}}(t^{m+1}) $.
It is useful to argue this also in a direct way:
By construction, the Hermite quadrature
formula underlying~\eqref{LDint-1-appr} is of order~$ m $, and
its error has the Peano representation (cf. also~\cite{auzingeretal13a})
\begin{equation} \label{quad-peano}
\frac{t}{m}\,\Theta_m(t,t) - \int_0^t \Theta_m(s,t)v\,\dd s =
\int_0^t \frac{s\,(t-s)^{m-1}}{m!}\,
\tfrac{\partial^{m}}{\partial s^{m}}\,\Theta_m(s,t)v\,\dd s.
\end{equation}
Here,
$ \tfrac{\partial^{m}}{\partial s^{m}}\,\Theta_m(s,t)v
= {\mathcal{O}}(1) $, because
$ \tfrac{\dd^{m}}{\dd s^{m}}\,D_m(s)v = {\mathcal{O}}(1) $
which follows from $ D_m(s)v = {\mathcal{O}}(s^{m-1}) $.
This shows that, indeed, the quadrature error~\eqref{quad-peano}
is $ {\mathcal{O}}(t^{m+1}) $.
Furthermore, a quadrature formula of order $ m+1 $ can be constructed by including
an additional evaluation of
\begin{equation*}
\tfrac{\partial}{\partial s}\,\Theta_m(s,t)v\big|_{s=t} = D_m^{[2]}(t)v,
\quad \text{with} \quad
D_m^{[2]}(t) = \tfrac{\dd}{\dd t} D_m(t) - \sca A\,D_m(t).
\end{equation*}
A routine calculation shows
\begin{equation}\label{eq:improvedHermiteQuad}
\int_0^t \Theta_m(s,t)v\,\dd s =
\frac{2\,t}{m+1}\,D_m(t)v
- \frac{t^2}{m(m+1)}\,D_m^{[2]}(t)v
+ {\mathcal{O}}(t^{m+2}),
\end{equation}
where the error depends on $ \tfrac{\dd^{m+1}}{\dd s^{m+1}}\,D_m(s)v = {\mathcal{O}}(1) $.
This may be considered as an improved error
estimate\footnote{In the setting of~\cite{auzingeretal13a}
                  (higher-order splitting methods) such an improved
                  error estimate was not taken into account since
                  it cannot be evaluated with reasonable effort in that context.}
which can be evaluated using
\begin{equation*}
\tfrac{\dd}{\dd t} D_m(t)v
= \sca^2 \tau_{m+1,m}\,e_m^\ast(T_m\,\ee^{\sca\,t\,T_m})e_1 v_{m+1}.
%&= \sca^2 \tau_{m+1,m}\,e_m^\ast T_m\,(\ee^{\sca\,t\,T_m}e_1) v_{m+1}
%&= \sca^2\,\tau_{m+1,m}\,
%   \frac{t^{m-1}}{(m-1)!}\,(e_m^\ast\,T_m^m\,e_1)\,v_{m+1} + {\mathcal{O}}(t^m).
\end{equation*}
%where $ \ee^{\sca\,t\,T_m}e_1 $ is available from the computed
%Krylov approximation~\eqref{Smv}.
With the solution in the Krylov subspace,
$\ee^{\sca\,t\,T_m} e_1$ with $e_m^\ast \ee^{\sca\,t\,T_m} e_1 = (\ee^{\sca\,t\,T_m} e_1)_m$,
we can compute the derivative of the defect at $\mathcal{O}(1)$ cost,
\begin{align*}
\tfrac{\dd}{\dd t} D_m(t)v
&= \sca^2 \tau_{m+1,m}\,e_m^\ast(T_m\,\ee^{\sca\,t\,T_m})e_1 v_{m+1} \\
&= \sca^2 \tau_{m+1,m}\,\big((T_m)_{m,m}\, (\ee^{\sca\,t\,T_m} e_1)_m + (T_m)_{m,m-1}\, (\ee^{\sca\,t\,T_m} e_1)_{m-1}\big)v_{m+1}.
\end{align*}
Also longer expansions may be considered, for instance
\begin{align*}
&\int_0^t \Theta_m(s,t)v\,\dd s =
\frac{3\,t}{m+2}\,D_m(t)v
- \frac{3\,t^2}{(m+1)(m+2)}\,D_m^{[2]}(t)v \\
& \qquad\qquad {} + \frac{t^3}{m(m+1)(m+2)}\,D_m^{[3]}(t)v
+ {\mathcal{O}}(t^{m+3}),
\quad \text{with} \quad
D_m^{[3]}(t) = \tfrac{\dd}{\dd t} D_m^{[2]}(t) - A\,D_m^{[2]}(t),
\end{align*}
etc.
}
\wa{This alternative way of computing improved error estimates is
worth investigating but will not be pursued further here.}

\paragraph{Quadrature estimate for~\eqref{errorapprox1} revisited.}
\tjnew{To continue from~\eqref{errorapprox1}, the nonexpansive case is required.}
%This section does not cover quadrature of the error via Cauchy's integral formula as given in~\cite{botchevetal13,FGS14a}.
In~\cite{lubich08} it is suggested to use the trapezoidal rule as a practical approximation
to the integral~\eqref{errorapprox1},
\begin{equation}\label{luquadtrapez}
\|L_m(t)v\|_2 \leq \tau_{m+1,m} \int_0^t |\delta_m(s)| \, \dd s \approx \tau_{m+1,m} \tfrac{t}{2} \,|\delta_m(t)|,
\end{equation}
or alternatively the Simpson rule.
Applying Hermite quadrature in~\eqref{errorapprox1} also directly leads to the error estimate~\eqref{eq:hermitquadnorm}.

For a better understanding of the approximation~\eqref{luquadtrapez} we consider the effective order of $|\delta_m(t)|$
as a function of~$ t $. Let us denote $f(t):=|\delta_m(t)|$ and assume $f(t)>0$ in a sufficiently small interval $(0,T]$.
For the Hermitian case this assumption is fulfilled for all $t>0$, see Proposition~\ref{upperboundphi}.

The effective order of the function $ f(t) $ can be understood as the slope of the double-logarithmic function
\begin{align*}
&\xi(\tau)=\ln(f(\ee^{\tau}))~~~\text{with}~~\tau=\ln t, \\
\text{with derivative} \quad &\xi'(\tau)=\frac{f'(\ee^{\tau})\,\ee^{\tau}}{f(\ee^{\tau})}.
\end{align*}
We denote the order by
\begin{equation}\label{def:effectiveorder}
\rho(t)=\frac{f'(t)\,t}{f(t)},
\end{equation}
and obtain
\begin{equation*}
f(t)=\frac{f'(t)\,t}{\rho(t)}.
\end{equation*}
Integration and application of the mean value theorem shows the existence of $t^\ast\in[0,t]$ with
\begin{equation*}
\int_0^t f(s)\,\dd s= \frac{1}{\rho(t^\ast)} \int_0^t f'(s)\,s\,\dd s,
\end{equation*}
and integration by parts gives
\begin{equation*}
\int_0^t |\delta_m(s)| \,\dd s = \frac{t\, |\delta_m(t)|}{1+\rho(t^\ast)}.
\end{equation*}
With the plausible assumption that the order is bounded by $1\leq \tilde{m}\leq \rho(t) \leq m-1 =\rho(0+)$ for $t\in[0,T]$, we obtain
\begin{equation}\label{defintbounds}
\tfrac{t}{m}\,|\delta_m(t)| \leq \int_0^t |\delta_m(s)| \,\dd s \leq \tfrac{t}{\tilde{m}+1}\,|\delta_m(t)| \leq \tfrac{t}{2}\,|\delta_m(t)|.
\end{equation}
The inequalities~\eqref{defintbounds}
show that the error estimate based on the trapezoidal quadrature~\eqref{luquadtrapez} leads to an upper bound of the error.
The error estimate based on Hermite quadrature \eqref{eq:hermitquadnorm} leads to a lower bound of the integral~\eqref{errorapprox1},
which not necessarily leads to a lower bound for the norm of the error.
\tjnew{In contrast to~\eqref{eq:hermitquadnorm}, the error estimate~\eqref{luquadtrapez} is not asymptotically correct for $ t \to 0 $.} 

\begin{remark}\label{remark:quadnew}
With $\rho(0+)=m-1$ and the assumptions that the effective order is slowly decreasing locally at $t=0$ and sufficiently smooth,
we suggest choosing $\tilde{m}=\rho(t)$ for a step of size $t$ \tjnew{to} improve the quadrature based estimate.
\begin{equation}\label{eq:effoquad}
\|L_m(t)v\|_2 \leq \tau_{m+1,m} \int_0^t |\delta_m(s)| \, \dd s \approx \tau_{m+1,m} \tfrac{t}{\rho(t)+1} \,|\delta_m(t)|,
\end{equation}
We will refer to this as \emph{effective order quadrature} estimate.
In the limit $t\to0$ this choice of quadrature is equivalent to the Hermite quadrature and, therefore, asymptotically correct.
\end{remark}

Up to now we did refer to the effective order of the defect $|\delta_m(t)|$. For $t \to 0$ the effective order of the error is given by $\rho(t)+1$.
%%%%%
%%%%%
%%%%%

\section{The matrix exponential as a time integrator.} \label{sec:timeint}
\tjnew{For simplicity we assume the nonexpansive case of~\eqref{exp(itA)v} in this section.}

We recall from~\cite{hoclub97} that superlinear convergence as a function of $ m $, the dimension of the underlying
Krylov space, sets in for
\begin{equation}\label{tcrit}
t\, \|\Opgen\|_2 \lessapprox m.
\end{equation}
%where $|\lambda_{\text{max}}|=\max_{\lambda\in\text{spec}(A)}|\lambda|$.
This relation also affects the error considered as a function of time $t$.
Equation~\eqref{tcrit} can be seen as a very rough estimate for
a choice of $t$ which leads to a systematic error and convergence behavior.
Only for special classes of problems as for instance symmetric negative definite matrices,
the relation~\eqref{tcrit} can be weakened, see~\cite{hoclub97,becrei09} for details.
%\tjnew{Results similar to~\eqref{tcrit} are also improved in~\cite{becrei09}.}
%Equation~\eqref{tcrit} holds for the skew-Hermitian case.
%\tj{Such a convergence behavior can also be observed in our error plots in section~\ref{sec:num} -- TODO: clarify or skip}.

In general a large time step~$t$ would necessitate large $m$ or a restart of the Krylov method.
However, choosing $m$ too large can lead to a deviation from
orthogonality of the Krylov basis due to round-off effects and is not recommended in general.
%A loss of orthogonality may lead to a loss of structure preserving properties of the matrix exponential.
Restarting the Krylov method is thus preferable.
\tjnew{Likewise, when memory issues must be taken into account, a restart is also advisable.}
For the matrix exponential seen as a time propagator, a simple restart is possible.
The following procedure has been introduced
in~\cite{sidje98} and is recapitulated here to fix the notation.

We split the time range $[0,t]$ into $N$ subintervals,
\begin{align*}
0&=t_0<t_1<\ldots<t_N=t, \\
\text{with step sizes}~~\Delta t_j&= t_{j}-t_{j-1},\quad j=1,\ldots,N.
\end{align*}
The exact solution at time $t_j$ is denoted by $ \mv{j} $, whence
\begin{equation*}
\mv{j} = E(\Delta t_{j})\mv{j-1}=E(t_{j})v,\quad\text{with}\quad \mv{0}=v.
\end{equation*}
For simplicity we assume that the dimension $m$ of the Krylov subspace is fixed
over the substeps.
We obtain approximations $ \mw{j }$ to $ \mv{j} $ by applying multiple restarted
Krylov steps, with orthonormal bases $ \Vmj $ and upper Hessenberg matrices $ \Tmj $.
Starting from $ \mw{0} = v $, for $ j=1,\ldots,N $,
\begin{align*}
\mw{j} := \Smj(\Delta t_{j}) \mw{j-1}
       = \Vmj \ee^{\sca\,\Delta t_{j} \Tmj} \big(\Vmj\big)^\ast \mw{j-1}
       = \Vmj \ee^{\sca\,\Delta t_{j} \Tmj}e_1.
\end{align*}
The error matrix in the $j$-th step is denoted by
\begin{equation*}
\Lmj(\Delta t_{j}) := E(\Delta t_{j})-\Smj(\Delta t_{j}),
\end{equation*}
and the accumulated error by
\begin{equation} \label{Lmstar}
L_m^{\star}(t)v=\mv{N}-\mw{N}.
\end{equation}
With
\begin{align*}
\mv{j}-\mw{j} &= E(\Delta t_j)\mv{j-1}-\Smj(\Delta t_j) \mw{j-1} \\
&= E(\Delta t_j)(\mv{j-1}-\mw{j-1}) +
%\big(E(\Delta t_j)-
\Lmj(\Delta t_j) \mw{j-1}
\end{align*}
we obtain
\begin{equation*}
L_m^{\star}(t)v = \sum_{j=1}^N E(\Delta t_N)\cdots E(\Delta t_{j+1})\,\Lmj(\Delta t_j) \mw{j-1}.
\end{equation*}
Recall our premise that $E(\cdot)$ is \wa{nonexpansive}
and assume that the local error is bounded by
\begin{equation}\label{localerror}
\|\Lmj(\Delta t_j)\mw{j}\|_2 \leq  \ttol \cdot \Delta t_j.
\end{equation}
Then,
\begin{equation*}
\|L_m^{\star}(t)v\|_2
\leq \sum_{j=1}^N \|\Lmj(\Delta t_j) \mw{j-1}\|_2 \leq  \ttol \sum_{j=1}^N \Delta t_j  = \ttol \cdot t.
\end{equation*}
The term $\|\Lmj(\Delta t_j)\mw{j}\|_2$ denotes the truncation error of a single substep and is studied in the first part of this paper.
We now apply local error estimates to predict acceptable time steps.

%%%%%
\paragraph{Step size control.}
For a single substep, the error estimate~\eqref{errornew}
suggests a step size to satisfy a given error tolerance $\ttol$ as
\begin{equation}\label{choosestepbyEra}
\Delta t_j = \left(\frac{\ttol\, \, m!}{\tau^{[j]}_{m+1,m}\gamma^{[j]}_m}\right)^{1/m},\quad j=1,\ldots,N.
\end{equation}
For a local error as in~\eqref{localerror}, we replace $\ttol$ by $ (\Delta t_j\,\,\ttol)$ in~\eqref{choosestepbyEra} and obtain
\begin{equation}\label{choosestepbyEralocal}
%\Delta t^{\star} =  \Big(\frac{\Delta t^{\star}\,\,\ttol\, \, m!}{\tau_{m+1,m}\gamma_m}\Big)^{1/m} \quad \Rightarrow \quad
\Delta t_j = \left(\frac{\ttol\, \, m!}{\tau^{[j]}_{m+1,m}\gamma^{[j]}_m}\right)^{1/(m-1)},\quad j=1,\ldots,N.
\end{equation}
%We will refer to the step size given by~\eqref{choosestepbyEralocal} as asymptotic step size.
We remark that $\Delta t_j$ can be computed together with the construction of the Krylov subspace,
therefore, $\tau^{[j]}_{m+1,m}$ and $\gamma^{[j]}_m$ are known values at this point.
For the corrected Krylov approximation $\Scor_m(t)\mv{j}$, see~\eqref{Smbar}, the error estimate given in~\eqref{eresphijcor} ($p=0$) suggests a local step size of
\begin{equation}\label{eq:stepasymptoticcorrected}
\Delta t_j = \left(\frac{\ttol\, \, (m+1)!}{\big\|\Opgen\mvmone{j}\big\|_2\,\tau^{[j]}_{m+1,m}\gamma^{[j]}_m}\right)^{1/m},\quad j=1,\ldots,N.
\end{equation}
The error estimator $\Erone$ and estimates given in Section~\ref{sec:defectquad} cannot be inverted directly
to predict the step size.
Computing a feasible step size is still possible via heuristic step size control.
This approach will be formulated for the error estimate $\Erone$ but can also be used in conjunction with all the error estimates given in Section~\ref{sec:defectquad}.
Ideas of heuristic step size control are given in~\cite{Gust91} in general and~\cite{sidje98} or~\cite{niewri12}
for a Krylov approximation of the matrix exponential.
For a step, with step size $\Delta t_{j-1}$ and estimated error $\Er^{[j-1]}$, a reasonable size for the
subsequent step can be chosen as
\begin{equation}\label{heuristicstepsize}
\Delta t_{j} = \left( \frac{\tjnew{\Delta t_{j}}\,\,\ttol}{\Er^{[j-1]}}\right)^{1/m} \Delta t_{j-1},\quad j=2,\ldots,N.
\end{equation}
In~\eqref{heuristicstepsize} we only need the evaluation of the error estimate for the previously computed step with step size $\Delta t_{j-1}$,
and we can substitute $\Erone$ for $\Er^{[j-1]}$.
%with safety factor $c=0.9$ in~\cite{sidje98} with reference to~\cite{Gust91}.
%\tj{
%For the skew-Hermitian case the safety factor is necessary but only because Expokit is rounding step sizes to 2 digits which leads to too large steps and break down of the method.
%Without safety factors and rounding of the step size we observe good convergence to steps with $\Er^{[j]}(\Delta t^{[j]})=\ttol$.
%Because of the very high order of the method, rounding can lead to large differences.
%}
The heuristic step size control always requires information from the previous step, therefore,
heuristic step size control cannot be used to compute the first time step $\Delta t_1$.
The first step is then usually based on a~priori estimates, which in general do not provide very sharp results.
The first step size in~\cite{sidje98} is chosen as
\begin{equation}\label{heuristicstepsize1}
\Delta t_1 = \frac{1}{\|H\|_{\infty}}\,\left(\frac{\ttol\,\, ((m+1)/\ee)^{m+1}\sqrt{2\,\pi\,(m+1)}}{4\,\|H\|_{\infty}}\right)^{1/m}.
\end{equation}

\begin{remark}\label{reqmark:stepsizeiteration}
%Following ideas can be used to improve the heuristic step size control for the case of $\Opgen$ Hermitian.
In the case of a Hermitian matrix $\Opgen$ the matrix $T_m$ is symmetric, tridiagonal and real-valued which allows cheap and robust computation of its eigenvalue decomposition.
The eigenvalue decomposition of $T_m$ is independent of the step size $\Delta t$ and allows cheap evaluation of $\Erone$ for multiple choices of $\Delta t$.
%Using the eigenvalue decomposition of $T_m$,
%the error estimate $\Erone$ can be computed with linear costs in $m$ for multiple choices of time steps $\Delta t$.
%Low costs on the computation of $\Erone$ for different time steps gives the possibility to correct the time step multiple times and also apply heuristic step size control for the first step.
In this case we can cheaply adapt the choice of $\Delta t_j$ in an iterative manner before continuing to time step $j+1$:
\begin{equation}\label{eq:stepsize2loops}
\begin{aligned}
\Delta t_{j,1}&:= \Delta t_{j-1}~~\text{or result of~\eqref{choosestepbyEralocal}},\\
%\Delta t_{j,1}&:= \Big(\frac{\ttol\, \, m!}{\tau^{[j]}_{m+1,m}\gamma^{[j]}_m}\Big)^{1/(m-1)}~~\text{by}~~\eqref{choosestepbyEralocal},\\
\Delta t_{j,l}&:= \Big( \frac{\tjnew{\Delta t_{j,l}}\,\,\ttol}{\Er^{[j,l-1]}}\Big)^{1/m} \Delta t_{j,l-1},~~~~~ l=2,\ldots,N_j,\\
\Delta t_{j}&:=\Delta t_{j,N_j}.\\
\end{aligned}
\end{equation}
By choosing $\Er^{[j,l-1]}$ as an error estimate for
the Krylov approximation of the $j$-th step
with time step $\Delta t_{j,l-1}$.
For $\Er^{[j,l-1]}$ we can use $\Erone$ or estimates given in Section~\ref{sec:defectquad}.
The aim of the iteration~\eqref{eq:stepsize2loops} is to determine a step size $\Delta t_{j,\infty}$ with $\Er^{[j,\infty]}=\ttol$.
The convergence behavior of iteration~\eqref{eq:stepsize2loops} depends on the structure of the corresponding error estimate.
In our practical examples this iteration converges monotonically in a small number of steps.

This step size iteration can also be used for the case of non-Hermitian $\Opgen$ with the drawback of higher computational cost for the error estimate $\Er^{[j,l-1]}$ for every $j$ and $l$.
\end{remark}

%%%%%%%%%%%%%%%%%%%%%%%%%%%%%%%%%%%%%%%%%%%%%%%%%%%%%%%%
%%%%%%%%%%%%%%%%%%%%%%%%%%%%%%%%%%%%%%%%%%%%%%%%%%%%%%%%

\section{Numerical considerations and examples}\label{sec:num}
\tjnew{
We give an illustration of our theoretical results
for two different skew-Hermitian problems in Subsection~\ref{sec:skew-herm},
a Hermitian problem in Subsection~\ref{sec:numherm}, 
and a non-normal problem in Subsection~\ref{sec:numgen}.
We also compare the performance of different error estimates
for practical step size control (Section~\ref{sec:timeint}) in Subsection~\ref{sec:skew-herm}.
To show that our error estimate~\eqref{asyerrest}
is efficient in practice we also compare
it with results delivered by the standard package Expokit~\cite{sidje98}
and a~priori error estimates.}
%%%%%

\subsection{The skew-Hermitian case}\label{sec:skew-herm}
\tjnew{
For our tests we use different types of matrices.}
\paragraph{Free Schr{\"o}dinger equation.}
We consider
\begin{equation} \label{Hschroe}
H=\tfrac{1}{4}\, \text{tridiag}(-1,2,-1) \in\R^{n\times n},
%\begin{pmatrix}
% 2&-1&  & &\\
%-1& 2&-1& &\\
%  &\ddots&\ddots&\ddots&\\
%  &  &-1 & 2 &-1\\
%  &  &   &-1 & 2
%\end{pmatrix}\in\R^{n\times n},
\end{equation}
with dimension $n=10\,000$.
The matrix $H$ is \tjnew{associated with a
finite difference or finite element discretization of the one-dimensional negative Laplacian}.
%\tjnew{The normalization factor $1/4$ in~\eqref{Hschroe} is chosen such that $ \text{spec}(H) \subseteq (0,1)$.}
%By rescaling time $t$ we can rescale the matrix $H$ arbitrarily.
%Therefore, without loss of generality,
%we are using the scalar factor $1/4$ to scale the eigenvalues of $H$
%such that $\lambda\in(0,1)$ for all $\lambda\in\text{spec}(H)$
%and ~\eqref{tcrit}
With $A=H$ and $\sca=-\ii$, in~\eqref{exp(itA)v} we obtain the free Schr{\"o}dinger equation.
The eigenvalue decomposition of $H$ is well known,
and we can use the discrete sine transform with high precision arithmetic in Matlab to compute the exact solution $E(t)v$, see~\eqref{exp(itA)v}. The starting vector $v$ is chosen randomly. To compute the Krylov subspace approximation $S_m(t)v$, see \eqref{Smv},
we use the eigenvalue decomposition of the tridiagonal matrix $T_m$.

\paragraph{Discrete Hubbard model.}
For the description of the Hubbard model we employ a self-contained notation.
The Hubbard model first appears in~\cite{hubbard63} and was further used in many papers and books, e.g.~\cite{Ma93,PKBS16}.
The Hubbard model is used to describe electron density on a given number of sites, which correspond to Wannier discretization of orbitals,
and spin up or down.
We consider the following Hubbard Hamiltonian, in second quantization and without chemical potential:
\begin{equation}\label{eq.HubHam}
H=\frac{1}{2} \sum_{i,j,\sigma} v_{ij} c_{j\sigma}^{\dagger} c_{i\sigma}
+ \sum_{j,\sigma} U \hat{n}_{j\sigma}\hat{n}_{j\sigma'},
\end{equation}
where $i,j$ sum over the number of sites $\nsit$ and the spins $\sigma,\sigma' \in\{\uparrow,\downarrow \}$ where $\sigma'$ is the opposite spin to $\sigma$.
%The notation is not consistent in the literature: $T_{ij}$ and $I$ is used in~\cite{hubbard63} instead of $v_{ij}$ and $U$.
%For the spin $\sigma$ the notation $\sigma=\pm 1$ is used in~\cite{hubbard63}.
The entries $v_{ij}$ with $i,j=1,\ldots,\nsit$ describe electron hopping from site $i$ to $j$.
In~\eqref{eq.HubHam}, the notation $c_{j\sigma}^\dagger c_{i\sigma}$ describes the $2$nd quantization operator
and $\hat{n}_{j\sigma}=c_{j\sigma}^\dagger c_{j\sigma}$ the occupation number operator.
For details on the notation in~\eqref{eq.HubHam} we can recommend several references, e.g.~\cite{hubbard63,Ja08,Ma93,PKBS16}.

For our tests we model $8$ electrons at $8$ sites ($\nsit=8$) with spin up and down for each site,
this leads to $16$ possible states for electrons.
Such an electron distribution is also referred to as half-filled in the literature.
%In opposite to~\cite[Section 3]{Ja08}
We also restrict our model by considering the number of electrons with spin up and down to be fixed as $\nsit/2$.
This leads to $n=(\text{binomial}(8,4))^2=4900$ considered occupation states which create a discrete basis.
For the numerical implementation of the basis we consider $16$-bit integers
for which each bit describes a position which is occupied in case the bit is equal to $1$ or empty otherwise.
The set of occupation states can be ordered by the value of the integers which leads to a unique representation
of the Hubbard Hamiltonian~\eqref{eq.HubHam} by a matrix $H\in\C^{n\times n}$.
Such an implementation of the Hubbard Hamiltonian is also described in~\cite[Section 3]{Ja08}.

In our test setting we use $U=5$ and parameter-dependent values for electron hopping $v_{ij}=v_{ij}(\omega)\in\C$ with $\omega\in(0,2\pi]$:
\begin{align*}
&v_{11}=v_{88}=-1.75,~~~v_{jj}=-2~~\text{for}~~j=2,\ldots,7,\\
&v_{j,j+1}=\bar{v}_{j+1,j}=-\cos\omega+\ii\,\sin\omega~~\text{for}~~j=1,\ldots,7~~\text{and}~~v_{ij}=0~~~\text{otherwise}.
\end{align*}
For this choice of $v_{ij}(\omega)$ we obtain \tjnew{a Hermitian} matrix $H_\omega\in\C^{n\times n}$
with $43980$ nonzero entries (for a general choice of $\omega$) and $\text{spec}(H_\omega)\subseteq (-19.1,8.3)$.
The spectrum of $H_\omega$ is independent of $\omega$.

A relevant application where the Hubbard Hamiltonian~\eqref{eq.HubHam} is of importance
is the simulation of oxide solar cells with the goal of finding candidates for
new materials promising a gain in the efficiency of the solar cell,
see~\cite{Held2007}.
The study of solar cells considers time-dependent electron hoppings $v_{ij}=v_{ij}(t)$
to model time-dependent potentials
which lead to Hamiltonian matrices $H(t)$.
\tjnew{The time-dependent Hamiltonian can be parameterized via~$\omega$.}
Time propagation of a linear, non-autonomous ODE system
can be approximated by Magnus-type integrators
which are based on one or more evaluations of matrix exponentials
applied to different starting vectors at several times $ t $, see \tjnew{for instance~\cite{blanesetal08b,blamoa05}}.
Our test setting for the Hubbard Hamiltonian with arbitrary $\omega$ is then obtained by~\eqref{exp(itA)v} with the matrix $A=H_\omega$ as described above and $\sca=-\ii$.

{In the following Subsection~\ref{sec:skew-herm} we focus on the skew-Hermitian case. For tests on the Hermitian case see Subsection~\ref{sec:numherm} below.}

\paragraph{Verification of upper \tjnew{error} bound.}
In the following Figures~\ref{fig:numexfreeSchro} and~\ref{fig:numexHub}
we compare the error $\|L_m(t)v\|_2$ with the error estimates
$\Erone$ and $\Era$. %see~\eqref{errorsaadterm1} and~\eqref{errornew}.
Figure~\ref{fig:numexfreeSchro} refers to
the matrix~\eqref{Hschroe} of the free Schr{\"o}dinger problem
and Figure~\ref{fig:numexHub}
to the Hubbard Hamiltonian~\eqref{eq.HubHam} with $\omega=0.123$.
For both cases we show results with Krylov subspace dimensions
$ m=10 $ and $ m=30 $, respectively.

We observe that the error estimate $\Erone$ is a good approximation \tjnew{to the error},
but it is not an upper bound in general.
In contrast, $\Era$ is a \tjnew{proven} upper error bound.
Up to round-off error, for $m=10$ we observe the correct asymptotic behavior of $\Era$ and $\Erone$.
For larger choices of $m$ the asymptotic regime starts
at time steps for which the error is already close to round-off precision.
Therefore, for larger choices of $m$, the Krylov approximation,
as a time integrator, cannot achieve its full order for typical time steps
in double precision.
%We will refer to this effect as a loss of local order.

\tjnew{{The matrix~\eqref{Hschroe}
has been scaled such that $\text{spec}(H)\subseteq(0,1)$ and $\|H\|_2\approx 1$.}
In accordance with~\eqref{tcrit} stagnation of the error
is observed for times $t \lessapprox m$,
see Figure~\ref{fig:numexfreeSchro}.}
%\wa{** Folgender Satz gehoert besser begruendet: **}
%\wa{** ? Ist Err1 asymptotisch korrekt fuer t to 0? **}
%\wa{This leads to the conclusion that the error estimate $\Era$ is a sharper estimate for small choices of $m$. Additionally, the error estimate $\Era$ is sharper for smaller time steps,
%such time steps usually are used if a small error is aimed for.}

%On any tested skew-Hermitian example we observed that the error $\|L_m(t)v\|_2$ behaves concave in a double $\log$ sense, the same holds for the defect $\delta_m(t)$.
%In practice we are interested in time steps for which the error is small enough, depending on the problem this can be $\|L_m(t)v\|_2\approx 10^{-14}$ to $\|L_m(t)v\|_2\approx 10^{-4}$.
%For such steps the local order of the method is lower than the theoretical order $m$. This effect was not always shown for the defect of not skew-Hermitian examples.
%%%%%

\begin{figure}
\centering
\begin{overpic}
[width=0.7\textwidth]{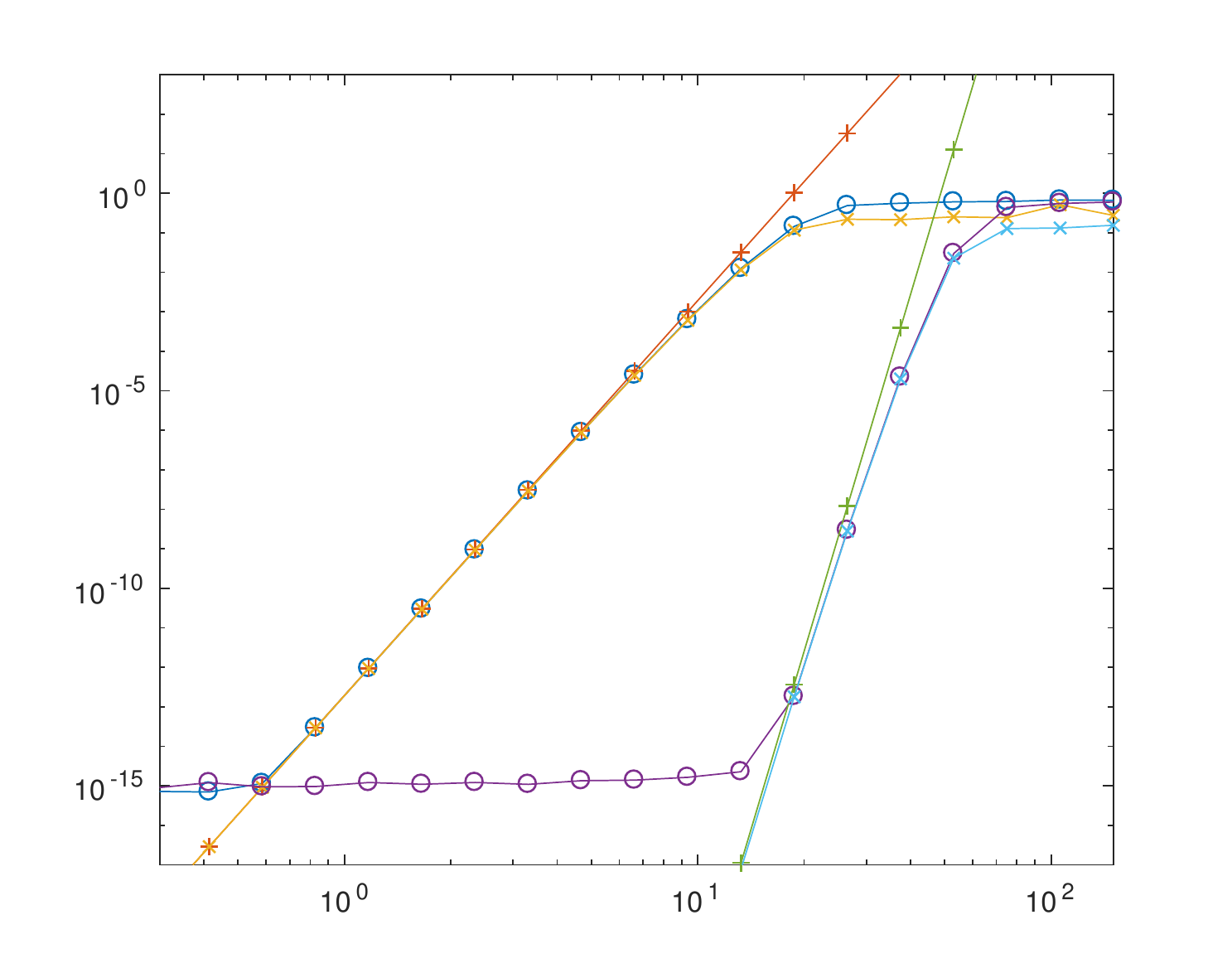}
\put(42,53){$m=10$}
\put(77,53){$m=30$}
\put(50,3){$t$}
\end{overpic}
\caption{Error $\|L_m(t)v\|_2$ ($\circ$) and the error estimates $\Erone$ ($\times$) and $\Era$ ($+$) for the free Schr{\"o}dinger problem and Krylov subspace dimensions $m=10$ and $m=30$.
$\Era$ is an upper bound for the error, and both estimates show the correct asymptotical behavior.
Due to round-off error, for $m=30$ the observed effective order is less clear than for $m=10$.}
\label{fig:numexfreeSchro}
\end{figure}

\begin{figure}
\centering
\begin{overpic}
[width=0.7\textwidth]{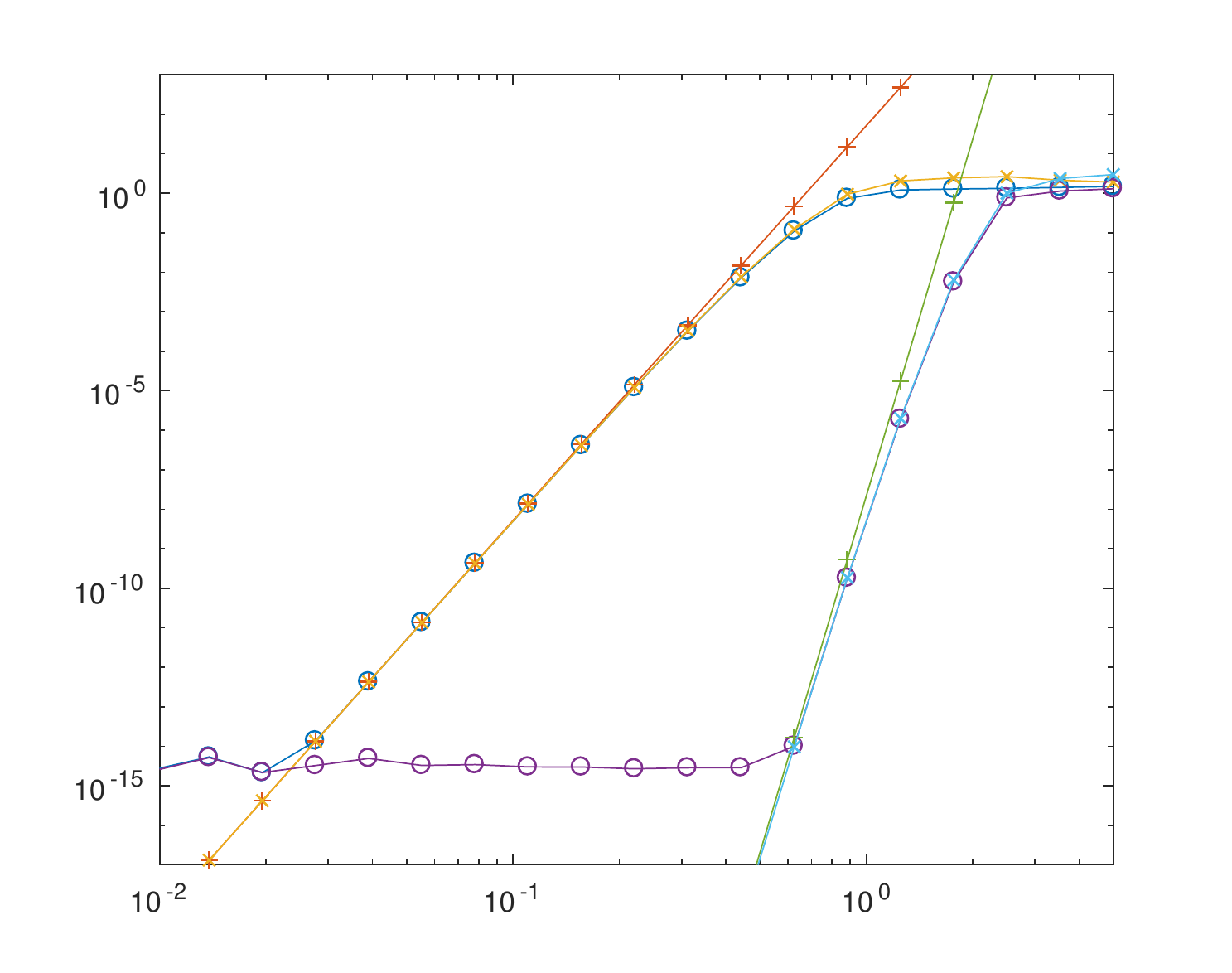}
\put(45,55){$m=10$}
\put(79.5,55){$m=30$}
\put(50,3){$t$}
\end{overpic}
\caption{Error $\|L_m(t)v\|_2$ ($\circ$) and the error estimates $\Erone$ ($\times$) and $\Era$ ($+$)
 for the Hubbard Hamiltonian with $\omega=0.123$
 and Krylov subspace dimensions $m=10$ and $m=30$.
This shows the same behavior as in Figure~\ref{fig:numexfreeSchro}.}
\label{fig:numexHub}
\end{figure}

%\FloatBarrier
We verify the error estimates in the skew-Hermitian setting of the free Schr{\"o}dinger equation~\eqref{Hschroe}
for the standard Krylov approximation of the $\phif_1$ function in Figure~\ref{fig:numexfreeSchrophi1}
and the corrected Krylov approximation of the matrix exponential function in Figure~\ref{fig:numexfreeSchrocor}.
In Figure~\ref{fig:numexfreeSchrophi1} the error estimator $\Erone$ refers to formula~\eqref{LPnextphierrorest}
and $\Era$ shows the \tjnew{upper error bound}~\eqref{eresphij} from Theorem~\ref{theorem:upperboundphi}, both for the case $p=1$.
In Figure~\ref{fig:numexfreeSchrocor}, $\Erone$ is from formula~\eqref{LPplusnextphierrorest}
and $\Era$ denotes the \tjnew{upper error bound}~\eqref{eresphijcor} from Theorem~\ref{theorem:asymerrorpricorrected}, both for the case $p=0$.
%\tj{$\phi_1$ $\Erone$ \eqref{LPnextphierrorest} $p=1$, $\Era$ Theorem~\ref{theorem:upperboundphi} eq~\eqref{eresphij} $p=1$
%corrected $\Erone$ \eqref{LPplusnextphierrorest} $p=0$, $\Era$ Theorem~\ref{theorem:asymerrorpricorrected} eq~\eqref{eresphijcor} $p=0$.}
\begin{figure}
\centering
\begin{overpic}
[width=0.7\textwidth]{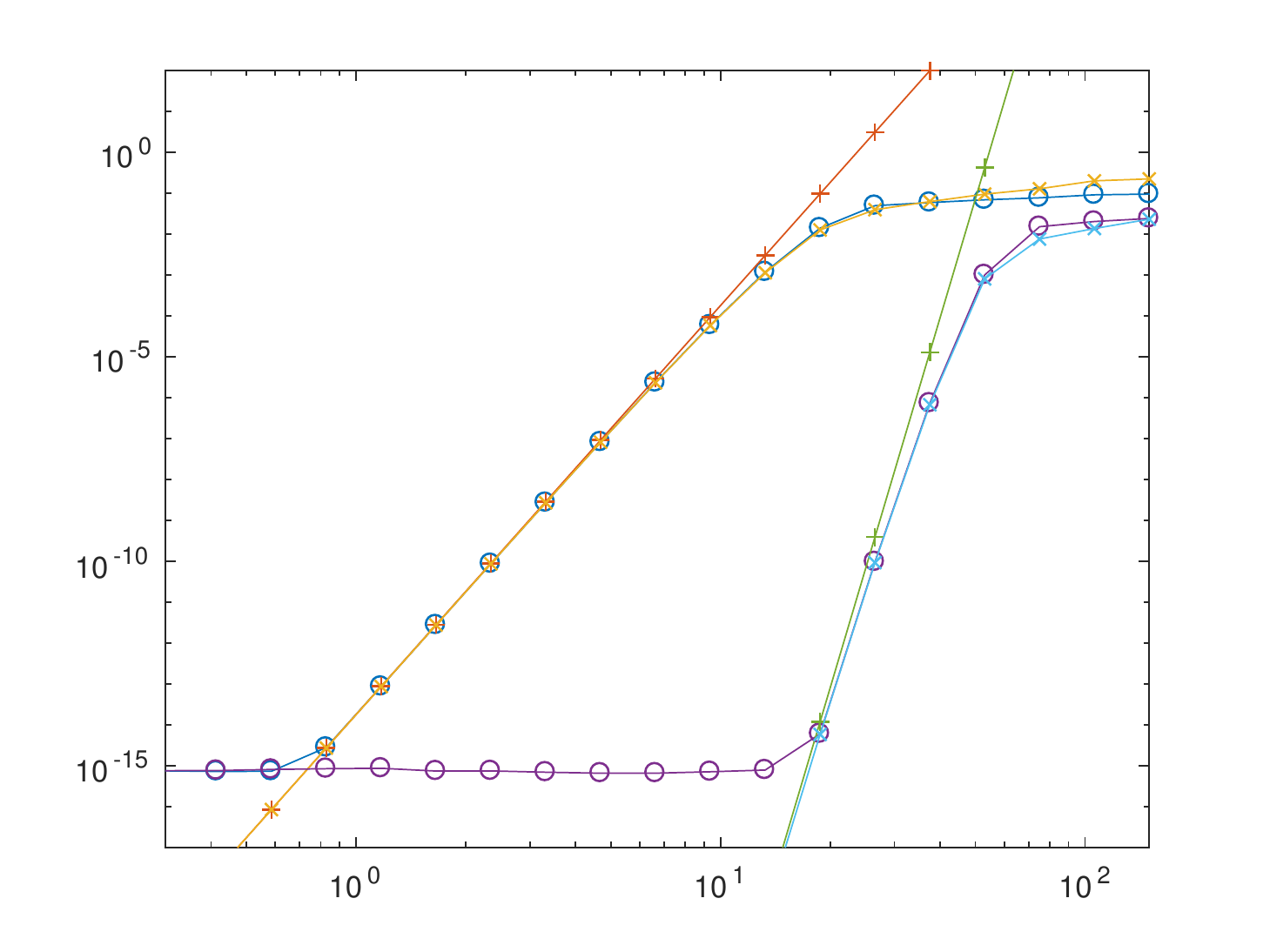}
\put(43,51){$m=10$}
\put(79,51){$m=30$}
\put(50,3){$t$}
\end{overpic}
\caption{Error $\|L_m^1(t)v\|_2$ ($\circ$) and the error estimates $\Erone$ ($\times$) and $\Era$ ($+$) for the free Schr{\"o}dinger problem and Krylov subspace dimension $m=10$ and $m=30$.}
%The error $L_m^1(t)v$ denotes the error of the standard Krylov approximation applied to the $\phi_1$ function, see~\eqref.
%The error estimate $\Era$ is given in~\eqref{LPnextphierrorest} for $p=1$ and does not lead to an upper bound in general.
%The asymptotical correctness of $\Era$ ~\eqref{eresphij} from Theorem~\ref{theorem:upperboundphi}
\label{fig:numexfreeSchrophi1}
\end{figure}

\begin{figure}
\centering
\begin{overpic}
[width=0.7\textwidth]{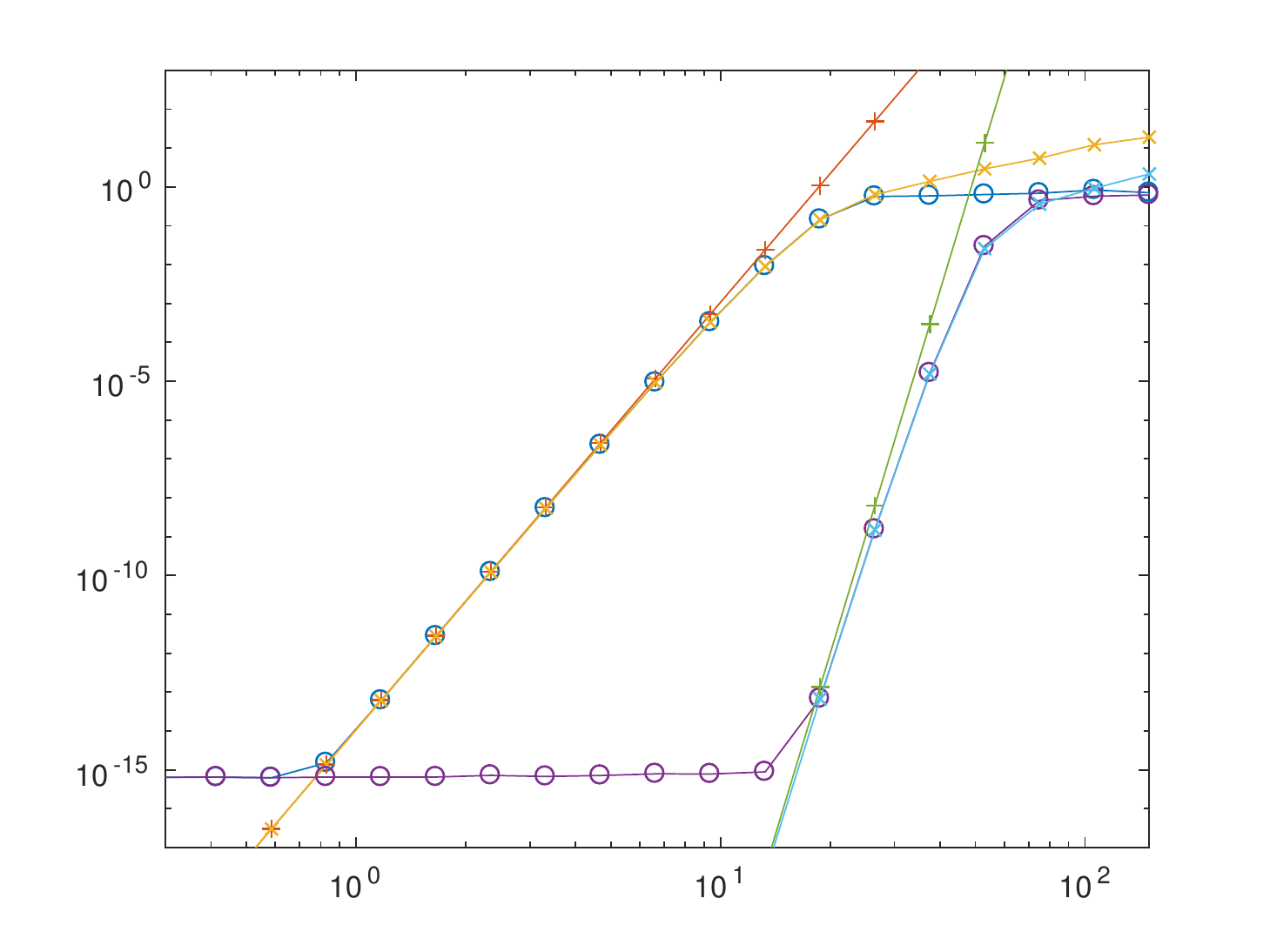}
\put(43,51){$m=10$}
\put(78.5,51){$m=30$}
\put(50,3){$t$}
\end{overpic}
\caption{Error $\|L^{+}_m(t) v\|_2$ ($\circ$) and the error estimates $\Erone$ ($\times$) and $\Era$ ($+$) for the free Schr{\"o}dinger problem and Krylov subspace dimension $m=10$ and $m=30$.}
\label{fig:numexfreeSchrocor}
\end{figure}

%%%%%

%\FloatBarrier
\paragraph{Illustration of defect-based quadrature error estimates from Section~\ref{sec:defectquad}.}

We first illustrate the performance of the estimates based on
Hermite quadrature according to~\eqref{eq:hermitquadnorm} and
improved Hermite quadrature according to~\eqref{eq:improvedHermiteQuad}
for the Hubbard model, see Figure~\ref{fig:num-quad1}.
%Remarks on the numerical computation of~\eqref{eq:improvedHermiteQuad} are given in Section~\ref{sec:defectquad}.
Both estimates are asymptotically correct, whereas the improved quadrature~\eqref{eq:improvedHermiteQuad}
is slightly better for larger time steps $t$,
with the drawback of one additional matrix-vector multiplication.
(See Remark~\ref{remark:cheapvsexpensive}
below for cost efficiency of more expensive error estimates.)

Figure~\ref{fig:num-quad2} refers to estimates based on
the trapezoidal rule~\eqref{luquadtrapez},
the effective order quadrature according to Remark~\ref{remark:quadnew},
and the Hermite quadrature~\eqref{eq:hermitquadnorm}.
For our test problems the assumptions from Section~\ref{sec:defectquad}
on the defect and its effective order
are satisfied for a significant range of values of~$ t $.
We also observe that the inequalities~\eqref{defintbounds}
are satisfied.
The effective order and Hermite quadrature estimates
behave in an asymptotically correct way, while the trapezoidal rule
estimate leads to an upper \tjnew{error} bound which is not sharp for $t\to 0$.

For the skew-Hermitian case $\sigma=-\ii$ and $A=H$ we obtain
\begin{align*}
&|\delta_m(t)|= \big((\ee^{\ii\,t\,T_m}e_1)_m(\ee^{-\ii\,t\,T_m}e_1)_m\big)^{1/2}~~~\text{and effective order $\rho(t)$, see~\eqref{def:effectiveorder},}\\
&\rho(t)=\frac{t\,\big(|\delta_m(t)|\big)'}{|\delta_m(t)|} =
\frac{\ii\,t\, (T_m)_{m-1,m}}{2}\bigg(
\frac{(\ee^{\ii\,t\,T_m}e_1)_{m-1}}{(\ee^{\ii\,t\,T_m}e_1)_m}
-\frac{(\ee^{-\ii\,t\,T_m}e_1)_{m-1}}{(\ee^{-\ii\,t\,T_m}e_1)_m}\bigg).
\end{align*}
For computing the effective order we only consider
time steps for which the defect is not too close to round-off precision,
$\rho(t)>0$, and where $\rho$ appears indeed to be
monotonically decreasing over the computed discrete time steps.
This restriction is compatible with
our assumptions in Section~\ref{sec:defectquad}.

\begin{figure}
\centering
\begin{overpic}
[width=0.7\textwidth]{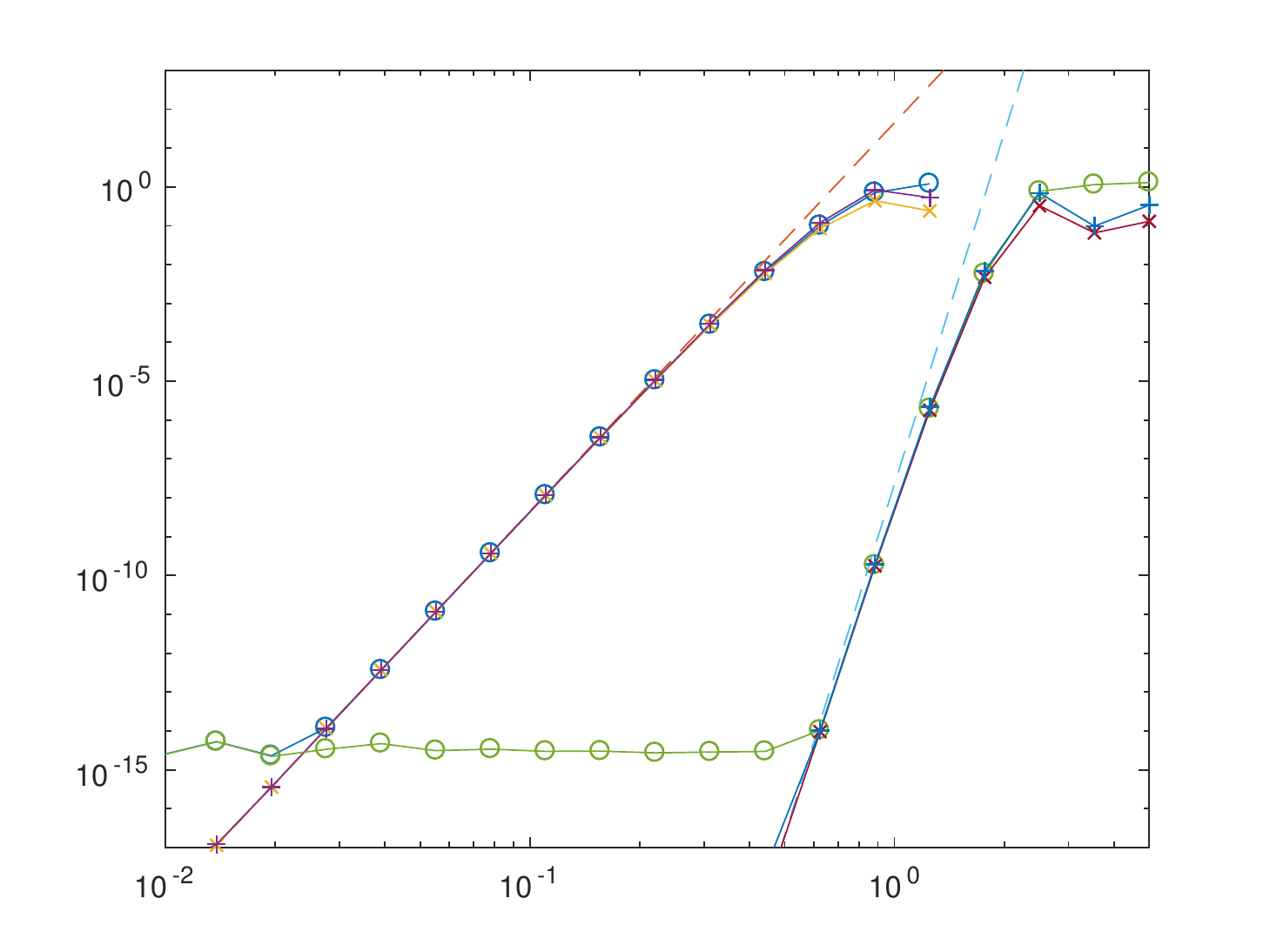}
\put(35,45){$m=10$}
\put(78,45){$m=30$}
\put(50,2){$t$}
\end{overpic}
\caption{Error $\|L_m(t) v\|_2$ ($\circ$) and the error estimates based on the Hermite quadrature ($\times$) and improved Hermite quadrature ($+$), see~\eqref{eq:hermitquadnorm} and~\eqref{eq:improvedHermiteQuad}, for the Hubbard Hamiltonian with $m=10$ and $m=30$.
The dashed lines show the error estimate $\Era$.}
\label{fig:num-quad1}
\end{figure}

\begin{figure}
\begin{tabular}{l}
\begin{overpic}
[width=0.7\textwidth]{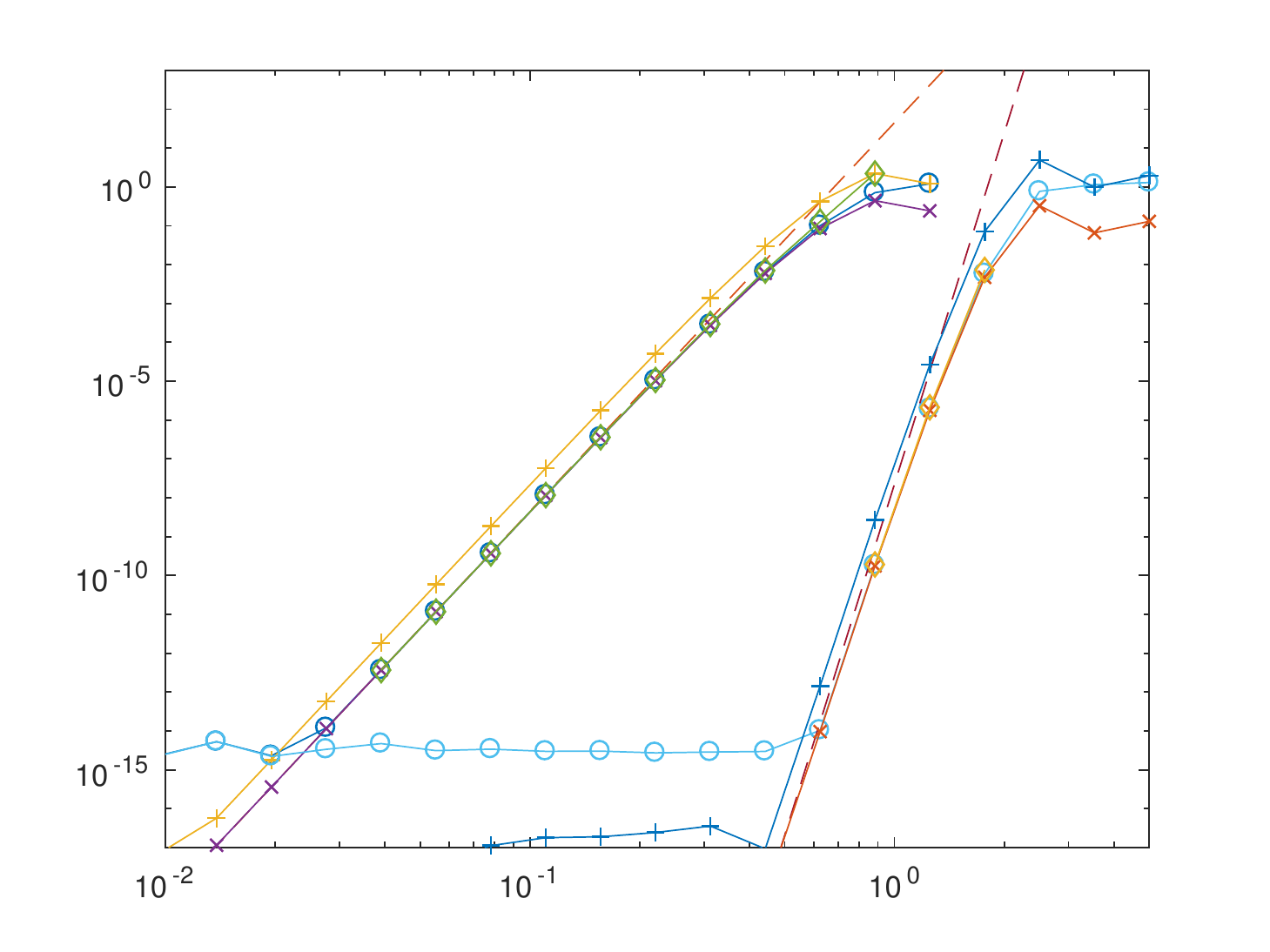}
\put(35,45){$m=10$}
\put(78,45){$m=30$}
\put(50,2){$t$}
\put(91,36.4){\includegraphics[width=0.31\textwidth]{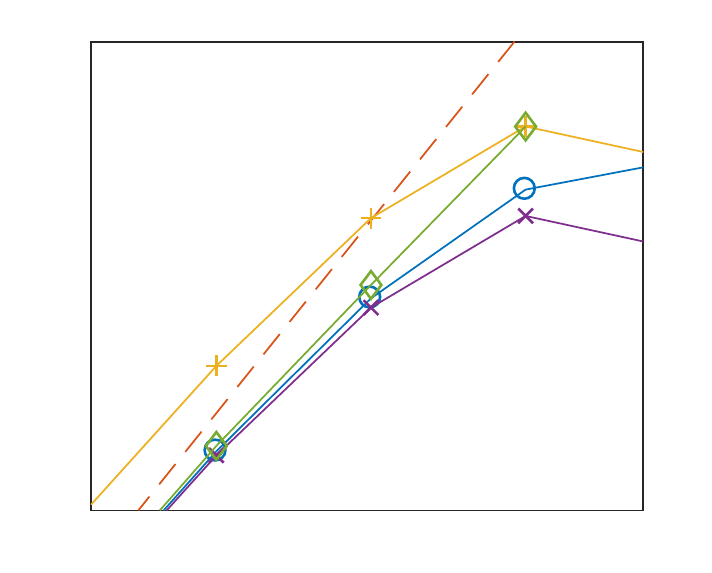}}
\put(91,4.2){\includegraphics[width=0.31\textwidth]{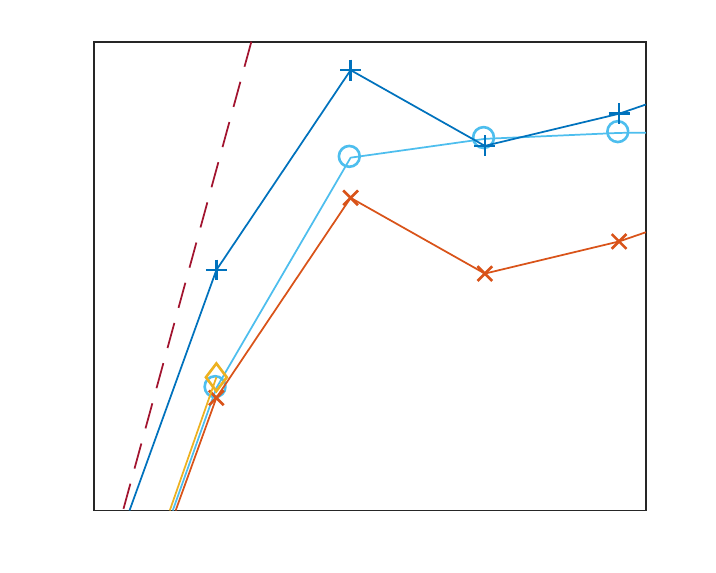}}
\put(120,41.4){\small$m=10$}
\put(120,9.2){\small$m=30$}
\end{overpic}
\end{tabular}
\begin{equation*}
\begin{aligned}
&\text{Computed effective order of the defect for $m=10$:}\\
&\begin{array}{|c|cccccccccc|}
\hline
t    &
3.9\cdot10^{-2} & 5.5\cdot10^{-2} & 7.8\cdot10^{-2} &
1.1\cdot10^{-1} & 1.5\cdot10^{-1} & 2.2\cdot10^{-1} &
3.1\cdot10^{-1} & 4.4\cdot10^{-1} & 6.3\cdot10^{-1} &
8.8\cdot10^{-1} \\\hline
\rho(t) & 8.99 & 8.98 & 8.95 & 8.90 & 8.81 & 8.63 & 8.24 & 7.44 & 5.66 & 1.00\\\hline
\end{array}\\
&\text{and $m=30$:}\\
&
\begin{array}{|c|ccc|}
\hline
t    & 8.8\cdot 10^{-1} & 1.2\cdot 10^0  &1.8\cdot 10^0 \\\hline
\rho(t) & 26.68 & 24.18 & 18.42\\\hline
\end{array}
\end{aligned}
\end{equation*}
\caption{The upper left plot shows the error $\|L_m(t) v\|_2$ ($\circ$) and the error estimates based on the Hermite quadrature~\eqref{eq:hermitquadnorm} ($\times$), the trapezoidal rule~\eqref{luquadtrapez} ($+$) and the effective order quadrature~\eqref{eq:effoquad} ($\Diamond$) for the Hubbard Hamiltonian with $m=10$ and $m=30$.
The dashed lines show the error estimate $\Era$.
On the right-hand side the graphics show a detail from the error plots to illustrate the inequalities~\eqref{defintbounds}.
The table on the bottom shows the computed effective order of the defect for $m=10$ and $m=30$ which is used for the effective order quadrature. }
\label{fig:num-quad2}
\end{figure}

\paragraph{Corrected Krylov approximation, mass conservation and loss of orthogonality.}
We remark that error estimates for the corrected Krylov approximation
usually require one additional matrix-vector multiplication, and applying a standard
Krylov approximation of dimension $m+1$ seems to
be a more favorable choice in our approach to error estimation.

The Krylov approximation of the matrix exponential
conserves the mass for the skew-Hermitian case
in contrast to the
corrected Krylov approximation.
Whether this is a real drawback of the corrected Krylov approximation
depends on the emphasis placed on mass conservation.
In the following examples we focus on the standard Krylov approximation,
with some exceptions which serve for comparisons with the original Expokit code, which is based on the corrected Krylov approximation.

In exact arithmetic we obtain mass conservation for the skew-Hermitian case:
For the case $\|v\|_2=1$ and standard Krylov approximation $S_m(t)v$ we have
\begin{equation}\label{eq:massconsv}
\|S_m(t)v\|_2 = \|V_m \ee^{-\ii\,t\,T_m}e_1\|_2 = e_1^\ast \ee^{\ii\,t\,T_m} V_m^\ast V_m\ee^{-\ii\,t\,T_m}e_1=1.
\end{equation}
The requirement $V_m^\ast V_m=I$ is essential to obtain mass conservation in~\eqref{eq:massconsv}.
In inexact arithmetic with larger choices of $m$ the loss of orthogonality of the Krylov basis $V_m$ is well known
and also observed for the Hubbard Hamiltonian, see Table~\ref{table:orthohubbard}.
\tjnew{Early studies on the effects of round-off errors for the Lanczos approximation of eigenvalue problems are given in \cite{Pa76}.
In the context of matrix functions we refer to \cite{druskinetal98,DK91,Gre89}.}
It was \tjnew{also} shown in~\cite{MMS18X} that the Lanczos approximation does not suffer critically from round-off errors for larger $m$.
Concerning mass conservation we are not aware of such stability results,
but \tjnew{our particular} examples still show relatively good mass conservation also for larger dimensions~$m$.

\begin{table}
\centering
\caption{Loss of orthogonality of the Krylov basis for the Hubbard model.}\label{table:orthohubbard}
\begin{equation*}
\begin{array}{|c|cccccccc|}\hline
m & 10 & 20 & 30 & 40 & 50 & 60 & 70 & 80\\\hline
\lfloor\log_{10}(\|V_m^\ast V_m-I\|_2)\rfloor & -14 & -14 & -13 & -11 & -10 & -8 & -6 & -5\\\hline
\end{array}
\end{equation*}
\end{table}

\begin{comment}

\tj{also remark that for too large $m$, mass conservation is no longer exact because of loss of orthogonality of $V_m$.}
Krylov subspace approximations are often used in the context of geometric integrators.
In this case a property like mass conservation is an important feature.
We note that mass conservation no longer holds for the
corrected Krylov approximation $\Scor_m(t)v$ defined in \tj{todo}~\eqref{Smbar}
in the sense that $\ee^{-\ii\,t\Tcor_m}$ is not unitary.
%Whether this is a real drawback of the asymptotic error estimate depends on the emphysis placed on mass conservation,
%which is theoretically only given for the standard Krylov method.
%We conclude that the asymptotic error estimate holds for the standard Krylov method, whereas the error estimate based on Corollary~\ref{co:saad92} can be adjusted to the corrected case.
Numerical tests of mass conservation of the standard and corrected Krylov approximation are shown in Table~\ref{table:masserror}.
We observe mass conservation for the standard Krylov approximation up to round-off precision.
For the corrected Krylov approximation a loss of mass conservation can be observed.
For small tolerances, the error in the mass is negligible for short integration times.
%Our numerical tests show that the error on the mass can negligible small if the error tolerance is chosen small enough.
%In general such the mass conservation can not be argued for the corrected Krylov method which leads to a disadvantage of the corrected method.
\end{comment}

\paragraph{Numerical tests for step size control.}
The idea of choosing discrete time steps for the Krylov approximation is described in Section~\ref{sec:timeint}.
The following tests are applied to the matrix exponential of the Hubbard Hamiltonian.
We first clarify the notation used for our test setting.
\begin{description}
\item {\em Expokit and Expokit$^{\star}$}.
The original Expokit code uses the corrected Krylov approximation with heuristic step size control and an
error estimator which is based on the error expansion~\eqref{sumofsaad}, see ~\cite[Algorithm 3.2]{sidje98} for details.
Since the standard Krylov approximation is not part of the Expokit package,
we have slightly adapted the code and its error estimate such that the standard Krylov approximation is used.
We refer to the adapted package as Expokit$^{\star}$.
With Expokit$^{\star}$ our comparison can be drawn with the standard Krylov approximation
which may in some cases be the method of choice as discussed above.
\item {\em Step size based on $\Era$.} In another test code the \tjnew{upper error bound} $\Era$ from Theorem~\ref{theorem:upperboundexp} is used.
With $\Era$ we obtain \tjnew{proven} upper bounds on the error and reliable step sizes~\eqref{choosestepbyEralocal}.
\item By {\em tr.quad}, {\em eff.o.quad}, and {$\mathit{Err}_1$} we refer to
the trapezoidal rule~\eqref{luquadtrapez}, the effective order quadrature~\eqref{eq:effoquad}, and~\eqref{errorsaadterm1}, respectively.
Because these error estimates cannot be inverted directly we need to apply heuristic ideas for the step size control, see~\eqref{heuristicstepsize}.
In addition, we use the iteration~\eqref{eq:stepsize2loops} given in Remark~\ref{reqmark:stepsizeiteration} to improve step sizes.
Monitoring the effective order $\rho(t)$ of the defect (see Section~\ref{sec:defectquad}) the heuristic step size control can be improved.
For the test problems we have solved, iteration~\eqref{eq:stepsize2loops}
converges in less than $2$~iterations for $m=10$ or less than $5$~iterations for $m=30$.
We simply choose $N_j=5$ for our tests.
%\item {A priori estimate~\rm{\cite[eq. (20)]{mohcar10}}}
%The error estimate given in~\cite[eq. (20)]{mohcar10} is used to approximate the local error.
%Numerically this can be done by adjusting the tolerance $\ttol \to (\ttol \Delta t)$ in an iterative manner.
%This error estimate requires spectral information, which may be a drawback.
%For tests on the Hubbard model we use $(\lambda_{\text{max}} - \lambda_{\text{min}}) = 27.4$ as suggested
%in the description of the Hubbard Hamiltonian.
%\item {A priori error estimate of Expokit}, see~\eqref{heuristicstepsize1}.
%\item {A priori estimate~\rm{\cite[Theorem 4]{hoclub97}}}
\item The {\em a~priori estimates~\eqref{heuristicstepsize1},~\cite[Theorem 4]{hoclub97} and~\cite[eq. (20)]{mohcar10}}
are given in the corresponding references.
Formula~\eqref{heuristicstepsize1} taken from the Expokit code directly provides a step size.
In~\cite[eq. (20)]{mohcar10} the computation of the step size is described.
For the error estimate given in~\cite[Theorem 4]{hoclub97} we apply Newton iteration to determine an appropriate step size.
{For tests on the Hubbard model we use $(\lambda_{\text{max}} - \lambda_{\text{min}}) = 27.4$ as suggested
in the description of the Hubbard Hamiltonian.}
\end{description}
In Remark~\ref{remark:cheapvsexpensive} below we also investigate the following variants:
\begin{description}
\item {\em Step size based on $\Era^+$.}
By $\Era^+$ we denote the \tjnew{upper error bound} for the corrected Krylov approximation as given in Theorem~\ref{theorem:asymerrorpricorrected} with $p=0$.
The corres\-ponding step size is given by~\eqref{eq:stepasymptoticcorrected}.
\item By {\em i.H.quad} we refer to the improved Hermite quadrature~\eqref{eq:improvedHermiteQuad}. Similarly to other quadrature error estimates we use heuristic step size control and iteration~\eqref{eq:stepsize2loops} to determine adequate step sizes.
\end{description}

\begin{remark}
In the Expokit code the step sizes are rounded to $2$ digits in every step.
Rounding the step size can give too large errors in some steps.
This makes it necessary to include safety parameters in Expokit which
on the other hand slow down the performance of the code.
It seems advisable to avoid any kind of rounding of step sizes.
\end{remark}

In Table~\ref{table:stepsize10steps} we compare the total time step $t$ for the Krylov approximation with $m=10$ and $m=30$
after $N=10$ steps obtained with the different step size control strategies.
For the local error we choose the tolerance $\ttol=10^{-8}$.
The original Expokit code seems to give larger step sizes,
but it also uses a larger number of matrix-vector multiplications, see Remark~\ref{remark:cheapvsexpensive}.
The error estimate $\Era$ leads to optimal step sizes for $m=10$ and close to optimal step sizes for $m=30$.
For any choice of $m$ the error estimate $\Era$ gives reliable step sizes.
The trapezoidal quadrature rule overestimates the error and, therefore, step sizes are smaller.
The effective order quadrature and $\Erone$ give optimal step sizes.
With the assumptions of Section~\ref{sec:defectquad} (which apply to our test examples), the trapezoidal and effective order quadrature give reliable step sizes.
For the error estimate $\Erone$ we do not have results on the reliability of the step sizes since the error estimate $\Erone$ does not lead to an upper bound of the error \tjnew{in general}.
The tested a priori estimates~\eqref{heuristicstepsize1},~\cite[Th. 4]{hoclub97}, and~\cite[(20)]{mohcar10}
overestimate the error and lead to precautious step size choices.
For all the tested versions the accumulated error
$ L_m^{\star} $ (see~\eqref{Lmstar}) satisfies $\|L_m^{\star} v\|_2/t \leq \ttol$.

\begin{table}
\centering
\caption{
The displayed step size $t$ is the sum of $N=10$ substeps computed by different versions of step size control,
as described above.
In the top table we show the results for $m=10$,
in the bottom table for $m=30$,
both for tolerance $\ttol=10^{-8}$, for the Hubbard Hamiltonian.
}
\label{table:stepsize10steps}
\begin{equation*}
\begin{aligned}
&\begin{array}{|c|c|c|c|c|c|c|c|c|c|}\hline
m=10
& \text{Expokit}
& \text{Expokit}^{\star}
& \Era
& \text{tr.quad}
& \text{eff.o.quad}
& \text{$\Erone$}
& \eqref{heuristicstepsize1}
& \text{\cite[Th. 4]{hoclub97}}
& \text{\cite[(20)]{mohcar10}}
\\\hline
t
&0.8930  &  0.6850  &  0.8422  &  0.7058  &  0.8443  &  0.8444  &  0.1918  &  0.4918  &  0.6879\\
 N
&10&10&10&10&10&10&10&10&10 \\
\text{\#\,m-v}
& 110 & 100 & 100 & 100 & 100 & 100 & 100 & 100 & 100\\\hline
\|L_m^{\star} v\|_2/t
& 3.4\cdot 10^{-09} & 3.1\cdot 10^{-09} & 9.8\cdot 10^{-09} & 2.0\cdot 10^{-09} & 1.0\cdot 10^{-08} & 1.0\cdot 10^{-08} & 2.4\cdot 10^{-14} & 7.8\cdot 10^{-11} & 1.6\cdot 10^{-09} \\\hline
% 9.9683e-09, 9.9824e-09
\end{array}&
\\
&\begin{array}{|c|c|c|c|c|c|c|c|c|c|}\hline
m=30
& \text{Expokit}
& \text{Expokit}^{\star}
& \Era
& \text{tr.quad}
& \text{eff.o.quad}
& \text{$\Erone$}
& \eqref{heuristicstepsize1}
& \text{\cite[Th. 4]{hoclub97}}
& \text{\cite[(20)]{mohcar10}}
\\\hline
t
& 8.5700  &  8.2500  &  9.7361  &  9.2582 &  10.2243  & 10.2338  &  2.1131  &  8.2642  &  8.8111 \\
 N
&10&10&10&10&10&10&10&10&10 \\
\text{\#\,m-v}
& 310 & 300 & 300 & 300 & 300 & 300 & 300 & 300 & 300\\\hline
\|L_m^{\star} v\|_2/t
& 2.4\cdot 10^{-10} & 2.8\cdot 10^{-10} & 2.6\cdot 10^{-09} & 7.0\cdot 10^{-10} & 9.5\cdot 10^{-09} & 9.7\cdot 10^{-09}& 2.9\cdot 10^{-15} & 3.3\cdot 10^{-11} & 1.9\cdot 10^{-10} \\\hline
%&2.4\, 10^{-10} & 2.8\, 10^{-10} & 2.6\, 10^{-9} & 7.0\, 10^{-10} & 9.5\, 10^{-9} & 9.7\, 10^{-9}& 2.9\, 10^{-15} & 3.3\, 10^{-11} & 1.9\, 10^{-10} \\\hline
\end{array}&
\end{aligned}
\end{equation*}
\end{table}

Apart from step size control, \tjnew{the upper error bound}~$\Era$ can be used on the fly %at every step of the Lanczos method
to test if the dimension of the Krylov subspace is already sufficiently large to solve the problem in a single time step with the required accuracy.
%This can lead to large improvements if multiple computations with small time steps are required, see Magnus integrators.
For our test problems this stopping criterion is applied to the $\Era$ estimate.%, but could be added to any variant of Lanczos.
We refer to Table~\ref{tablet03m30}, in which we observe the Krylov method with error estimate $\Era$ to stop after $17$ steps
instead of computing the full Krylov subspace of dimension $30$.
In comparison, the original Expokit package needs a total of $62$ matrix-vector multiplications.

%For the example of the small time step $t=0.3$ and $m=30$ (see Table~\ref{tablet03m30})
%testing $\Era$ leads to large improvements in the computational cost.
%The first time step of the Expokit code is chosen by a priori error estimates which underestimate the time step and,
%therefore, a restart is needed in Expokit and $62$ matrix-vector multiplications are used in total.
%Applying the error estimate~\eqref{asyerrest} during the Lanczos method already detects sufficient accuracy
%in the first step with a Krylov subspace of dimension $17$
%requiring $17$ matrix-vector multiplications in total.
%The result of the other codes is similar,
%mainly the error is accepted to be small enough after
%one step which still leads to the computation of the full Krylov subspace $m=30$.

\begin{table}
\centering
\caption{With a test setting similar to Table~\ref{table:stepsize10steps}, we now compute up to a fixed time $t=0.3$ and choose the number $ N $~of steps according to the step size control.
We use a tolerance $\ttol=10^{-8}$ and $m=30$.
For this problem we see a significant reduction
in the number of matrix-vector multiplications used for the estimate $\Era$ by the stopping critera described in the text.}
\label{tablet03m30}
\begin{equation*}
\begin{array}{|c|c|c|c|c|c|c|c|c|}\hline
{m=10}
& \text{Expokit}
& \text{Expokit}^{\star}
& \Era
& \text{tr.quad}
& \Erone
& \eqref{heuristicstepsize1}
& \text{\cite[Th. 4]{hoclub97}}
& \text{\cite[(20)]{mohcar10}}
\\\hline
t
&0.3&0.3&0.3&0.3&0.3&0.3&0.3&0.3 \\
 N
&2&2&1&1&1 &2&1&1\\
\text{\#\,m-v}
& 62 & 60 & 17 & 30 & 30 & 60 & 30 & 30\\\hline
\|L_m^{\star} v\|_2/t
& 8.4\cdot 10^{-15} & 8.4\cdot 10^{-15} & 1.0\cdot 10^{-09}  & 9.7\cdot 10^{-15} &
  9.7\cdot 10^{-15} & 1.0\cdot 10^{-14} & 9.7\cdot 10^{-15} & 9.7\cdot 10^{-15} \\\hline
\end{array}
\end{equation*}
\end{table}

\begin{remark}\label{remark:cheapvsexpensive}
Error estimates for the corrected Krylov approximation or improved error estimates such as the improved Hermite quadrature~\eqref{eq:improvedHermiteQuad}
require additional matrix-vector multiplications.
Instead of investing computational effort in improving the error estimate,
one may as well increase the dimension of the standard Krylov subspace.
For comparison we test the original Expokit code, the corrected Krylov approximation with error estimate $\Era^+$
and the improved Hermite quadrature~\eqref{eq:improvedHermiteQuad} with Krylov subspace $m-1$.
Table~\ref{Table:ExpensivErrEstvsCheapErrEst} shows that a standard Krylov approximation with dimension $m$ leads to better results,
although all considered versions use the same number
of matrix-vector multiplications.
Since the reliability of error estimates such as $\Era$
has been demonstrated earlier,
it appears that additional cost to improve the error estimate
is not justified.

\begin{table}
\centering
\caption{
All variants shown use exactly $m$ matrix-vector multiplications.
Whereas Expokit, improved Hermite quadrature (i.H.quad) and $\Era^+$ imply higher cost for the error estimate,
the other codes $\Era$, effective order quadrature (eff.o.quad) and $\Erone$ use standard Krylov subspaces and do not spend additional matrix-vector multiplications on error estimates.
%since reliabilty is also given for some cheap estimates such as $\Era$.
%A direct comparison of the error estimate $\Era$ for the standard and corrected Krylov approximation shows superior results for the standard Krylov approximation.
}
\label{Table:ExpensivErrEstvsCheapErrEst}
\begin{equation*}
\begin{aligned}
\begin{array}{|c|c|c|c|c|c|c|}\hline
m=10& \text{Expokit}
& \Era^+
& \text{i.H.quad}
& \Era
& \text{eff.o.quad}
& \Erone
\\\hline
t
&0.6620&0.7828&0.5863&0.8346&0.8366&0.8368 \\
N
& 10 & 10 & 10 & 10 & 10 & 10  \\\hline
\text{\#\,m-v}
& 100 & 100 & 100 & 100 & 100 & 100 \\\hline
\|L_m^{\star} v\|_2/t
& 4.1\cdot 10^{-09} & 8.8\cdot 10^{-09} & 1.0\cdot 10^{-08} & 9.8\cdot 10^{-09} & 1.0\cdot 10^{-08} & 1.0\cdot 10^{-08} \\\hline
\end{array}
\\
\begin{array}{|c|c|c|c|c|c|c|}\hline
m=30& \text{Expokit}
& \Era^+
& \text{i.H.quad}
& \Era
& \text{eff.o.quad}
& \Erone
\\\hline
t
&  8.1900  &  9.5763  &  9.6591  &  9.7482  & 10.2378  & 10.2473 \\
N
& 10 & 10 & 10 & 10 & 10 & 10  \\\hline
\text{\#\,m-v}
& 100 & 100 & 100 & 100 & 100 & 100 \\\hline
\|L_m^{\star} v\|_2/t
& 3.6\cdot 10^{-10} & 2.7\cdot 10^{-09} & 9.2\cdot 10^{-09} & 2.6\cdot 10^{-09} & 9.5\cdot 10^{-09} & 9.7\cdot 10^{-09} \\\hline
\end{array}
\end{aligned}
\end{equation*}
\end{table}
\end{remark}

\subsection{The Hermitian case}\label{sec:numherm}

To obtain a more complete picture, we also briefly consider the case of \tjnew{a Hermitian} matrix $A=H$
%$M$ in \eqref{eq0}, correspoding to
with $\sigma=1$ in \eqref{exp(itA)v}.
Such a model is typical of the discretization of a parabolic PDE.
Thus, the result may depend on the regularity of the initial data,
which is chosen to be random in our experiments.

\tjnew{
\paragraph{Heat equation.}
To obtain the heat equation in~\eqref{exp(itA)v} we choose $A=H$ in~\eqref{Hschroe} and $\sca=-1$.
Details on the test setting are already given in Subsection~\ref{sec:skew-herm}.}

For the heat equation, $H$ given in~\eqref{Hschroe}, we can also verify the error estimates, see Figure~\ref{fig:numexheat}.
In comparison to the skew-Hermitian case we do not observe a large time regime for which the error is of the asymptotic order $m$.
As shown in Proposition~\ref{upperboundphi} we do obtain \tjnew{an upper error bound} using $\Erone$ for the heat equation.
However, the evolution is not unitary but nonexpansive in this case, whence the asymptotics are not observed as clearly here.

Similarly to the skew-Hermitian case, we can also apply the effective order quadrature according to Remark~\ref{remark:quadnew} to the Hermitian case.
For the Hermitian case results of Proposition~\ref{upperboundphi} can be applied to obtain
\begin{align*}
&|\delta_m(t)| = \delta_m(t) = (\ee^{\ii\,t\,T_m}e_1)_m~~~\text{and effective order $\rho(t)$, see~\eqref{def:effectiveorder},}\\
&\rho(t)=\frac{t\,\big(|\delta_m(t)|\big)'}{|\delta_m(t)|} =
t\,\bigg( \frac{(T_m)_{m,m-1}\,(\ee^{t\,T_m}e_1)_{m-1}}{(\ee^{t\,T_m}e_1)_m} + (T_m)_{m,m}\bigg).
\end{align*}
For computing the effective order we only consider
time steps for which the defect is not too close to round-off precision,
$\rho(t)>0$, and where $\rho$ appears indeed to be
monotonically decreasing over the computed discrete time steps.
This restriction is compatible with
our assumptions in Section~\ref{sec:defectquad}.

\begin{figure}
\centering
\begin{overpic}
[width=0.7\textwidth]{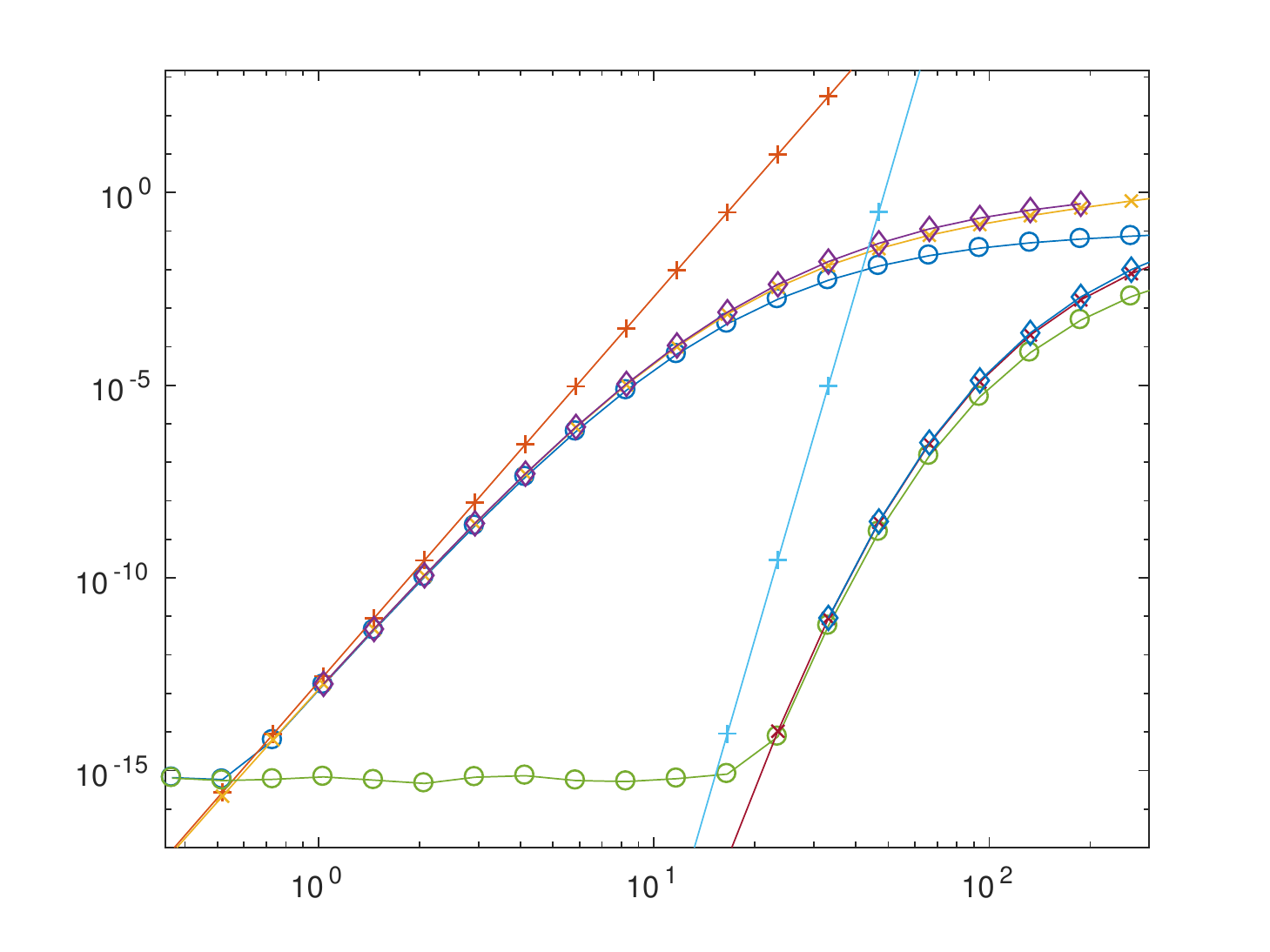}
\put(25,39){$m=10$}
\put(78,39){$m=30$}
\put(50,2){$t$}
\end{overpic}
\begin{equation*}
\begin{aligned}
&\text{Computed effective order of the defect for $m=10$ (partly):}\\
&\begin{array}{|c|ccccccccccc|}
\hline
t    &
1.0\cdot10^{0} & 1.5\cdot10^{0} & 2.1\cdot10^{0} &
2.9\cdot10^{0} & 4.1\cdot10^{0} & 5.9\cdot10^{0} &
8.3\cdot10^{0} & 1.2\cdot10^{1} & 1.7\cdot10^{1} &
2.3\cdot10^{1} & 3.3\cdot10^{1}\\\hline
\rho(t) &
8.50 & 8.30 & 8.02 &
7.64 & 7.14 & 6.48 &
5.65 & 4.66 & 3.58 &
2.52 & 1.60\\\hline
\end{array}\\
&\text{and $m=30$:}\\
&\begin{array}{|c|ccccccc|}
\hline
t    &
3.3\cdot10^{1} & 4.7\cdot10^{1} & 6.6\cdot10^{1} &
9.4\cdot10^{1} & 1.3\cdot10^{2} & 1.9\cdot10^{2} &
2.7\cdot10^{2}\\\hline
\rho(t) &
16.60 & 13.47 & 10.33 &
7.48 & 5.15 & 3.36 &
2.07 \\\hline
\end{array}
\end{aligned}
\end{equation*}
\caption{Error $\|L_m(t)v\|_2$ ($\circ$), the error estimates $\Erone$ ($\times$) and $\Era$ ($+$)
and the error estimate based on the effective order quadrature~\eqref{eq:effoquad} ($\Diamond$)
for the heat equation with $m=10$ and $m=30$.
The tabular on the bottom shows some of the computed values for the effective order.}
\label{fig:numexheat}
\end{figure}

\tjnew{\subsection{A non-normal problem}\label{sec:numgen}
For a more general case we consider a convection-diffusion equation (see \cite{MN01,EE06}).
\begin{equation}\label{convdifpde}
\begin{aligned}
&\partial_t u = \Delta u - \tau_1 \partial_{x_1} u - \tau_2 \partial_{x_2} u,~~\tau_1,\tau_2\in\R,~~~u=u(t,x),~~t\geq 0,~~ x\in\Omega=[0,1]^3,\\
&u(0,x)=v(x)~~\text{for}~x\in\Omega,~~u(t,x)=0~~\text{for}~x\in\partial\Omega.
\end{aligned}
\end{equation}
Following \cite{MN01,EE06} we use a central finite difference scheme to discretize the partial differential operator in~\eqref{convdifpde}.
The grid is chosen uniformly with $(n+2)^3$ points and mesh width $h=1/(n+1)$.
The dimension $N$ of the discrete operator is $N=n^3$.
Choosing $n=15$ we obtain $N=3357$.
The discretized operator is given by
\begin{align}
&A=( I_{n\times n} \otimes (I_{n\times n} \otimes C_1) ) 
+ ( B \otimes I_{n\times n} + I_{n\times n}\otimes C_2 ) \otimes I_{n\times n} )\in\R^{N\times N},~~~\text{with}\label{Aconvdif}\\
&~~~B=\tfrac{1}{h^2} \text{tridiag}(1,-2,1)\in\R^n,~~~
C_i=\tfrac{1}{h^2} \text{tridiag}(1+\mu_i,-2,1-\mu_i)\in\R^n,~~i=1,2,\notag
\end{align}
and $\mu_i=\tau_i\,(h/2)$. 
The spectrum of the non-normal matrix $A$ in~\eqref{Aconvdif} (see \cite{MN01}) satisfies
\begin{align*}
\text{spec}(A)\subseteq &\tfrac{1}{h^2}\,[-6-2\cos(\pi\,h)\text{Re}(\theta),-6+2\cos(\pi\,h)\text{Re}(\theta)] \\
   &~\times \tfrac{1}{h^2}\,\ii\,[-2\cos(\pi\,h)\text{Im}(\theta),2\cos(\pi\,h)\text{Im}(\theta)] .
\end{align*}
with $\theta = 1 + \sqrt{1-\mu_1^2} + \sqrt{1-\mu_2^2}$.
Therefore, the eigenvalues are complex-valued if at least one $\mu_i>1$.
The matrix $A$ depends on the parameters $\mu_i$, correspondingly $\tau_i$, for which we consider two different cases,
\begin{equation}\label{cdsetting1}
\text{$\mu_1=0.9$, $\mu_2=1.1$},~~~\text{with}~~\text{spec}(h^2\,A) \subseteq [-9,-3] \times \ii[-1,1],
\end{equation}
and
\begin{equation}\label{cdsetting2}
\text{$\mu_1=\mu_2=10$},~~~\text{with}~~\text{spec}(h^2\,A) \subseteq [-8,-4] \times \ii [ -39,39].
\end{equation}

In the following numerical experiments we apply the Krylov approximation to $\ee^{t\,A}v$
($\sigma=1$ in~\eqref{exp(itA)v}) for different time steps $t$ and starting vector $v=(1,\ldots,1)^\ast\in\R^{N}$ as in~\cite{MN01}.
Since $A$ is non-normal we use the Arnoldi method to generate the Krylov subspace.

The error estimates $\Era$ and $\Erone$ are compared to the exact error norm $\|L_m(t) v\|_2$ in Figure~\ref{fig:numexcd0911} for the case~\eqref{cdsetting1}
and in Figure~\ref{fig:numexcd1010} for the case~\eqref{cdsetting2}.
As shown in Theorem~\ref{theorem:upperboundexp} the error estimate $\Era$ constitutes an upper error bound.
The error estimate $\Erone$ gives a good approximation of the error but has not been proven to give an upper bound in general.

Compared to~\eqref{cdsetting2}, the spectrum for~\eqref{cdsetting1} is closer to the Hermitian case.
The spectrum for~\eqref{cdsetting2}, on the other hand, is dominated by large imaginary parts
similarly as in the skew-Hermitian case.

In Figure~\ref{fig:numexcd0911} we observe effects similar to the Hermitian case.
The asymptotic order $m$ of the error does not hold for a large time regime,
and the error estimate $\Era$ is not as sharp as in the skew-Hermitian case.
On the other hand, in Figure~\ref{fig:numexcd1010}, we observe that the performance of the error estimates is closer to the skew-Hermitian case.
Therefore, the upper error bound $\Era$ is sharp for a larger range of time steps.
As already observed for the Hermitian and skew-Hermitian cases, the error of the Krylov approximation is closer to its asymptotic order $m$
for smaller choices of $m$.
%The exact behaviour of the error estimates $\Era$ obviously depends on the starting vector and can not be fully predicted in general.
% and explicit spectral information
}

\begin{figure}
\centering
\begin{overpic}
[width=0.7\textwidth]{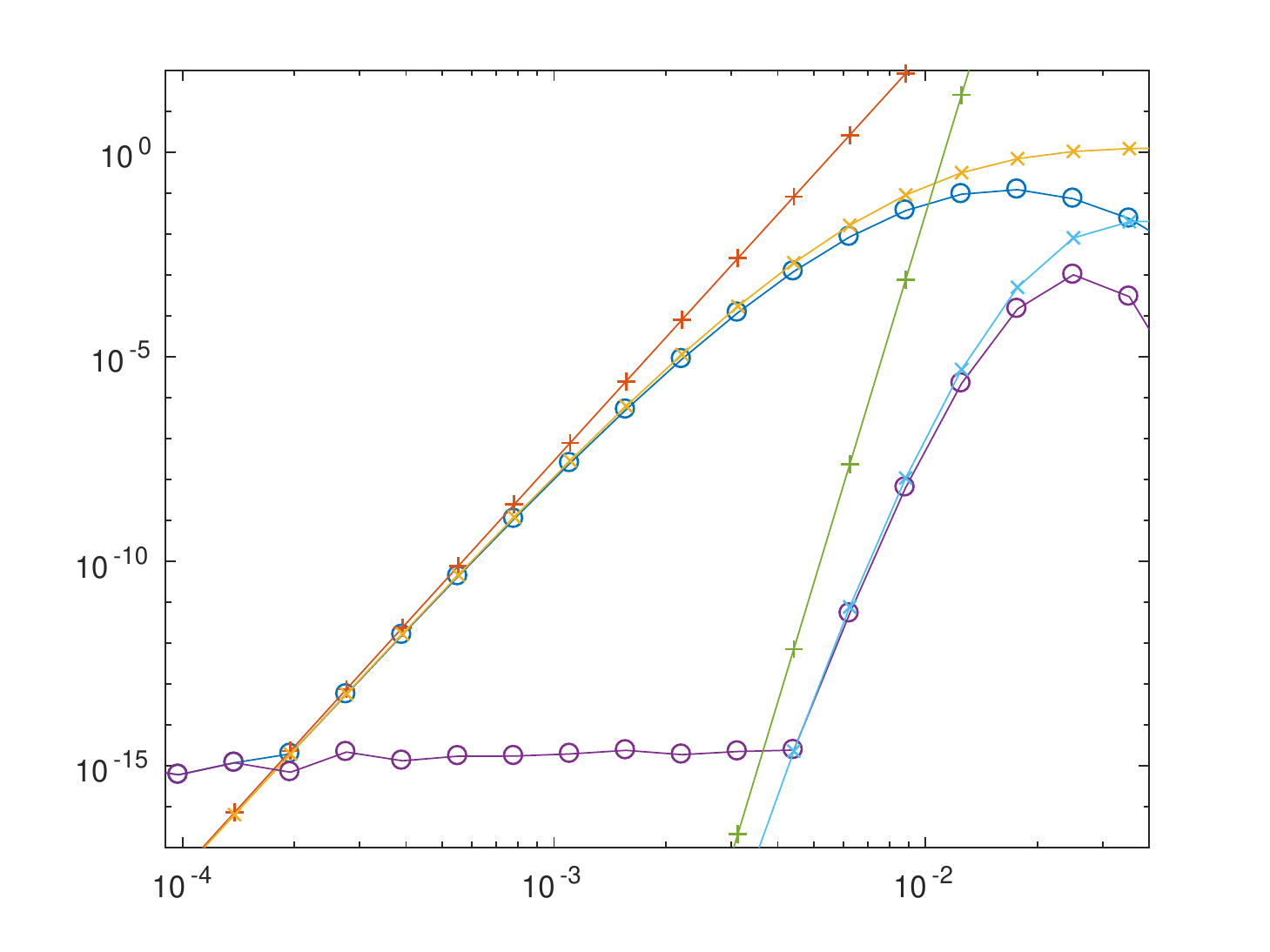}
\put(25,39){$m=10$}
\put(78,39){$m=30$}
\put(50,2){$t$}
\end{overpic}
\caption{\tjnew{Error $\|L_m(t)v\|_2$ ($\circ$) and the error estimates $\Erone$ ($\times$) and $\Era$ ($+$) for the convection-diffusion problem~\eqref{Aconvdif} with $\mu_1=0.9$ and $\mu_2=1.1$
 and Krylov subspace dimensions $m=10$ and $m=30$.}}
\label{fig:numexcd0911}
\end{figure}

\begin{figure}
\centering
\begin{overpic}
[width=0.7\textwidth]{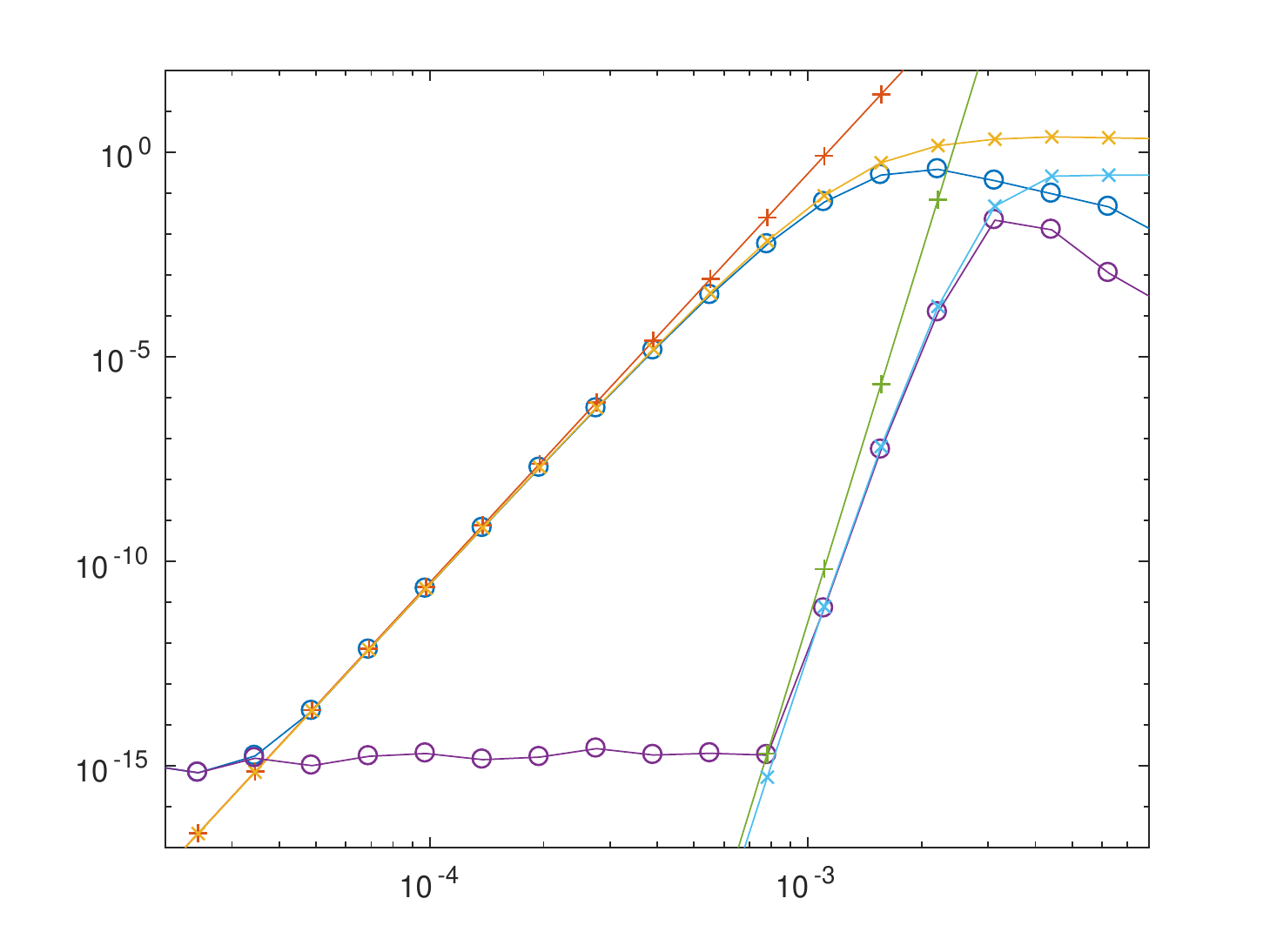}
\put(25,39){$m=10$}
\put(78,39){$m=30$}
\put(50,2){$t$}
\end{overpic}
\caption{\tjnew{Error $\|L_m(t)v\|_2$ ($\circ$) and the error estimates $\Erone$ ($\times$) and $\Era$ ($+$) for the convection-diffusion problem~\eqref{Aconvdif} with $\mu_1=\mu_2=10$
 and Krylov subspace dimensions $m=10$ and $m=30$.}}
\label{fig:numexcd1010}
\end{figure}

\section{Summary and outlook.}
We have studied a new reliable error estimate $\Era$
for Krylov approximations to the matrix exponential and $ \phif $-functions.
\tjnew{This error estimate constitutes an upper bound on the error},
and it can be computed on the fly at nearly no additional cost.
The \tjnew{Krylov} process can be stopped as soon as the error estimate satisfies a given tolerance.

Step size control for a simple restarted scheme is an important application.
The \tjnew{upper error bound $\Era$} is an appropriate tool for this task,
since the optimal step size for a given tolerance can be computed directly.
This is not the case for other error estimates
for the Krylov approximation, which usually employ
heuristic schemes to compute optimal step sizes in the restarting approach.
Also the use of a~priori bounds is not optimal in most cases.
Comparing \tjnew{our step size control} with heuristic versions shows that our procedure
allows larger step sizes \tjnew{in relevant cases} using reliable error estimates.
In addition to better performance,
we can avoid safety parameters and
assumptions on the convergence of an error expansion.
\tjnew{Therefore, our approach provides a reliable strategy} to find the optimal step size.
Numerical examples illustrate our theoretical results.
% and show that error estimates are sharp
%in the case of smaller Krylov subspace dimensions and tolerances.
%This leads to good performance of the error estimate $\Era$ in relevant cases. }
%For larger dimensions we observe reduction in the local order.
%We observed that for $t$ small enough, the local order is monotonically decreasing.

%This loss of local order will be the topic of further investigation
%and can possibly lead to enhanced defect-based error estimates.
The shift-and-invert method is also a relevant approach
which deserves further investigations.

\bigskip\noindent
{\textbf{Acknowledgements.}}
This work was supported by the Doctoral College TU-D,
Technische Universit{\"a}t Wien,
and by the Austrian Science Fund (FWF) under the grant P 30819-N32.

%\bibliography{books,appl,num,schroedinger}{}
%\bibliographystyle{plain}
%\bibliographystyle{chicago}

\end{document}